\documentclass[12pt,a4paper,reqno]{amsart}

\usepackage[T1]{fontenc}
\usepackage[utf8]{inputenc}

\usepackage{amsmath,amssymb,mathtools,mathrsfs}
\usepackage{geometry}
\usepackage[shortlabels]{enumitem}
\usepackage{microtype}
\usepackage{comment}

\usepackage[
colorlinks=true,
linkcolor=blue,
citecolor=blue,
urlcolor=blue
]{hyperref}

\usepackage[
backend=biber,
style=numeric-comp,
sorting=nyt,
sortcites=true,
giveninits=true,
maxbibnames=99,
doi=false,
url=false,
isbn=false,
eprint=false
]{biblatex}

\AtEveryBibitem{%
	\clearfield{doi}%
	\clearfield{url}%
	\clearfield{urldate}%
	\clearfield{isbn}%
	\clearfield{issn}%
}

\DeclareFieldFormat[article,inproceedings,incollection]{title}{#1}

\allowdisplaybreaks
\numberwithin{equation}{section}

\theoremstyle{plain}
\newtheorem{theorem}{Theorem}[section]
\newtheorem{proposition}[theorem]{Proposition}
\newtheorem{lemma}[theorem]{Lemma}
\newtheorem{corollary}[theorem]{Corollary}
\theoremstyle{definition}

\theoremstyle{remark}
\newtheorem{remark}[theorem]{Remark}

\newcommand{\ii}{\mathrm{i}}
\newcommand{\R}{\mathbb R}
\newcommand{\C}{\mathbb C}
\newcommand{\N}{\mathbb N}
\newcommand{\E}{\mathbb E}
\newcommand{\Cov}{\operatorname{Cov}}
\newcommand{\Dom}{\operatorname{Dom}}
\newcommand{\Span}{\operatorname{span}}
\newcommand{\HS}{\mathrm{HS}}

\title[Optimal Gaussian approximation for complex Wiener--It\^o integrals]
{Optimal Gaussian approximation for complex Wiener--It\^o integrals}

\author{Huiping Chen}
\address{Huiping Chen, Academy of Mathematics and Systems Science, Chinese Academy of Sciences, Beijing 100190, China}
\email{chenhp@amss.ac.cn}

\author{Yong Chen}
\address{Yong Chen, School of Big Data, Baoshan University, Baoshan, Yunnan 678000, China}
\email{zhishi@pku.org.cn; chenyong77@gmail.com}

\author{Yong Liu}
\address{Yong Liu, LMAM, School of Mathematical Sciences, Peking University, Beijing 100871, China}
\email{liuyong@math.pku.edu.cn}

\subjclass[2020]{Primary 60F05, 60H05; Secondary 60H07, 62M09}
\keywords{Complex Wiener--It\^o integrals, density approximation, Sobolev norms, non-smooth probability metrics, Malliavin calculus, complex fractional Ornstein--Uhlenbeck process}

\thanks{H. Chen is supported by the China Postdoctoral Science Foundation under Grant Number 2024M763480. Y. Chen is supported by NSFC (No. 12461029) and the PhD Research Startup Fund of Baoshan University (No. BSKY2540). Y. Liu is supported by NSFC (No. 12231002) and the Center for Statistical Science, PKU}

\begin{document}
	
	\begin{abstract}
		We study density regularity and optimal Gaussian approximation for complex multiple Wiener--It\^o integrals. We first give a kernel criterion for absolute continuity and derive explicit formulas for the density and its derivatives, together with finite-order Sobolev estimates. Our main quantitative result shows that the sum of an intrinsic non-Gaussian term and the covariance mismatch gives the optimal rate simultaneously in every fixed Sobolev norm and in total variation, Kolmogorov and $1$-Wasserstein distances. The intrinsic term is determined by two complex third-order moments and the complex fourth-moment defect, and is equivalently described by complex kernel contractions. No additional Malliavin non-degeneracy assumption is imposed in these asymptotic results. As an application, we obtain explicit optimal rates for the normalized numerator in the least-squares estimator of a complex fractional Ornstein--Uhlenbeck process.
	\end{abstract}
	\maketitle
	
	\section*{Notations}
	\begin{center}
		\begin{tabular}{@{}r p{0.68\textwidth}@{}}
			$\mathfrak H$ :& a complex separable Hilbert space with conjugation $h\mapsto\overline h$.\\
			$\mathfrak H_{\R}$ :& the real Hilbert space $\{h\in\mathfrak H:\overline h=h\}$.\\
			$\mathfrak K$ :& the real direct sum $\mathfrak H_{\R}\oplus\mathfrak H_{\R}$.\\
			$\mathfrak U$ :& a generic real separable Hilbert space used to formulate Malliavin calculus; later specialized to $\mathfrak K$. \\
			$Z$ :& a complex isonormal Gaussian process over $\mathfrak H$.\\
			$W$ :& the associated real isonormal Gaussian process over $\mathfrak K$.\\
			$G$ :& a generic real isonormal Gaussian process over $\mathfrak U$, used to formulate Malliavin calculus; later specialized to $W$. \\
			$\mathscr H_{p,q}(Z)$ :& the $(p,q)$-th complex Wiener--It\^o chaos of $Z$.\\
			$\mathcal H_r(W)$ :& the $r$-th real Wiener--It\^o chaos of $W$.\\
			$I_{p,q}(f)$ :& the complex $(p,q)$-th Wiener--It\^o integral of $f\in\mathfrak H^{\odot p}\otimes\mathfrak H^{\odot q}$ with respect to $Z$.\\
			$I_r^W(u)$ :& the real $r$-th Wiener--It\^o integral of $u\in\mathfrak K^{\odot r}$ with respect to $W$.\\
			$f^\star$ :& the reversed complex conjugate of a complex kernel $f\in\mathfrak H^{\odot p}\otimes\mathfrak H^{\odot q}$.\\
			$\widetilde F$ :& $(\mathrm{Re}\,F,\mathrm{Im}\,F)$ for a complex random variable $F$.\\
			$D,\delta$ :& the Malliavin derivative and divergence operators associated with $G$, later specialized to $W$.
		\end{tabular}
	\end{center}
	
	\setcounter{tocdepth}{2}
	\tableofcontents
	
	\section{Introduction}\label{sec:Introduction}
	
	Complex Gaussian analysis has a long history and arises naturally in a variety of problems in mathematical physics, stochastic partial differential equations, statistics, and signal processing. It\^o introduced complex multiple Wiener--It\^o integrals and the corresponding complex chaos decomposition through complex Hermite polynomials in \cite{ito1952complex}. Since then, the theory has developed in connection with complex Ornstein--Uhlenbeck operators, product formulas, independence criteria, fourth-moment phenomena, and quantitative Gaussian approximation; see, for example, \cite{campese2015fourth,ChenChenLiu2025Berry,Chen2017,ChenLiu2014,chen2017fourth,chenliu2019}. Applications include cosmic microwave background polarization \cite{Kamionkowski_1997,MARINUCCI2008585,marinucci_peccati_2011}, stochastic complex Ginzburg--Landau equations \cite{RevModPhys.74.99,CCL24,Hoshino2017}, statistical inference for complex Ornstein--Uhlenbeck processes \cite{AACZ2024,ChenHuWang2017}, and complex Gaussian noise models \cite{Aghaei2008,BARONE2005,Matalkah2008,Reisenfeld2003}.
	
	A complex random variable can be identified with a two-dimensional real vector, but this identification does not by itself capture the structure of a complex Wiener--It\^o integral. If $F=I_{p,q}(f)$ is a complex Wiener--It\^o integral, then the bidegree $(p,q)$, the reversed conjugate kernel $f^\star$, and the complex contractions of $f$ and $f^\star$ all play an essential role. At the level of moments, the natural quantities are not arbitrary cumulants of $(\mathrm{Re}\,F,\mathrm{Im}\,F)$, but the complex moments
	\[
	\E[F^3],\qquad \E[F^2\overline F],
	\]
	together with the fourth-moment defect
	\[
	\Delta_{4,\C}(F)
	:=
	\E|F|^4-2\bigl(\E|F|^2\bigr)^2-|\E[F^2]|^2.
	\]
	This distinction is important for the present paper. The real representation of a complex multiple integral gives access to sharp results from the real theory, but turning those results back into an intrinsic complex statement requires a separate identification of the relevant moments and contractions.
	
	This identification has been a recurring theme in our previous work. The complex-to-real representation of kernels was made explicit in \cite{ccl22}; the intrinsic complex moment scale for optimal smooth Gaussian approximation was identified in \cite{Chen2024Optimal}; and the role of complex contractions in Berry--Ess\'een estimates was further developed in \cite{ChenChenLiu2025Berry}; see also the overview \cite{ccl25}. The present paper combines these ingredients with the recent optimal real-vector theorem of Chen, Chen and Liu \cite{CCL2026}. In particular, the passage from the real theorem to the complex result below is not a formal application of realification. One must first prove that the real rate, built from third-order joint cumulants and coordinate fourth cumulants of $(\mathrm{Re}\,F,\mathrm{Im}\,F)$, is equivalent to a small collection of intrinsic complex moments. One must then identify that same scale through the $(p,q)$-contraction structure of the complex kernel.
	
	Our first contribution concerns density regularity. We give a kernel criterion for absolute continuity of a single complex multiple integral and derive explicit formulas for its density and its derivatives. To control derivatives up to any prescribed order, it is enough to assume finitely many Malliavin--Sobolev norms and one sufficiently high but finite negative moment of the Malliavin determinant. In the asymptotically Gaussian regime, the required negative moments are supplied automatically by the superconvergence theorem of Herry, Malicet and Poly \cite{HMP2024}.

	We now state the main quantitative result. We use the Sobolev norm
	$\|\cdot\|_{W^{m,r}(\R^2)}$ and the Kolmogorov, total variation,
	and $1$-Wasserstein distances $d_{\mathrm{Kol}}$, $d_{\mathrm{TV}}$,
	and $d_{\rm W}$, respectively; see \eqref{eq:sobolev-norm} and \eqref{eq:probability-distances} for their definitions. For a positive definite covariance matrix $\Sigma$, $\phi_\Sigma$ denotes the density of $N_2(0,\Sigma)$, and $\|\cdot\|_{\HS}$ denotes the Hilbert--Schmidt norm; see \eqref{eq:gaussian-density-HS}. For a complex random variable $F$ with finite fourth moment, set
	\begin{equation}\label{eq:intro-MC}
		M_{\C}(F):=
		\max\left\{
		|\E[F^3]|,\ |\E[F^2\overline F]|,\ \Delta_{4,\C}(F)
		\right\}.
	\end{equation}
	For a complex multiple Wiener--It\^o integral, $\Delta_{4,\C}(F)\ge0$; see, for example, \cite[Lemma~4.3]{Chen2024Optimal}. For two non-negative quantities $a_\tau$ and $b_\tau$ indexed by a parameter $\tau$, we write $a_\tau\asymp b_\tau$ if there exist constants $0<c\le C<\infty$ such that $cb_\tau\le a_\tau\le Cb_\tau$ for all sufficiently large $\tau$. We write $\xrightarrow{\mathrm{law}}$ for convergence in distribution.

	\begin{theorem}\label{thm:intro-main}
		Fix $p,q\in \N\cup\{0\}$ with $p+q\ge2$. For $n\ge1$, let $F_n=I_{p,q}(f_n)$, where $f_n\in\mathfrak H^{\odot p}\otimes\mathfrak H^{\odot q}$, and set
		\[
		\widetilde F_n=(\mathrm{Re}\,F_n,\mathrm{Im}\,F_n),
		\qquad
		\Sigma_n=\Cov(\widetilde F_n).
		\]
		Assume that
		\[
		\Sigma_n\to\Sigma,
		\qquad
		\widetilde F_n\xrightarrow{\mathrm{law}}N_2(0,\Sigma),
		\]
		where $\Sigma$ is positive definite. Then, for every $m\in\N\cup\{0\}$ and $r\in[1,\infty]$, $\widetilde F_n$ has a density $p_{F_n}\in W^{m,r}(\R^2)$ for all sufficiently large $n$ and
		\begin{equation}\label{eq:intro-fixed-target}
			\|p_{F_n}-\phi_\Sigma\|_{W^{m,r}(\R^2)}
			\asymp
			d_\ast\bigl(\widetilde F_n,N_2(0,\Sigma)\bigr)
			\asymp
			M_{\C}(F_n)+\|\Sigma_n-\Sigma\|_{\HS},
		\end{equation}
		where $d_\ast\in\{d_{\mathrm{TV}},d_{\mathrm{Kol}},d_{\rm W}\}$. The comparison constants may depend on $p,q,m,r$ and the limiting covariance matrix $\Sigma$, but not on $n$.
	\end{theorem}
	
	Theorem~\ref{thm:intro-main} has three features that are worth emphasizing. First, the same rate holds for total variation, Kolmogorov and $1$-Wasserstein distances and for every fixed Sobolev order. Thus passing from distributional approximation to density approximation, or to derivatives of the density of any prescribed finite order, causes no loss of rate. If the Gaussian target is chosen to match the covariance $\Sigma_n$, the covariance term disappears and the intrinsic approximation scale is exactly $M_{\C}(F_n)$; see Theorem~\ref{thm:general-optimal}. This result strengthens previous quantitative results for complex multiple Wiener--It\^o integrals. Chen \cite[Proposition~4.2]{Chen2024Optimal} identified $M_{\C}(F_n)$ as the optimal Gaussian approximation scale under a smooth distance for the covariance-matched Gaussian target, while Chen, Chen and Liu \cite[Theorem~3.1]{ChenChenLiu2025Berry} proved that there exist $0<c\le C<\infty$, independent of $n$, such that
	\begin{align*}
		cM_{\C}(F_n)
		\le
		d_{\rm W}\bigl(\widetilde F_n,N_2(0,\Sigma_n)\bigr)
		\le
		C\sqrt{\Delta_{4,\C}(F_n)}.
	\end{align*}
	Theorem~\ref{thm:general-optimal} shows that $M_{\C}(F_n)$ itself is
	the optimal scale for the covariance-matched approximation, simultaneously in Sobolev density norms and in $d_{\rm TV}$, $d_{\rm Kol}$ and $d_{\rm W}$.
	
	Second, the fixed-target estimate in \eqref{eq:intro-fixed-target} is not obtained from the covariance-matched estimate by the triangle inequality alone. The upper bound is straightforward, but the lower bound requires ruling out cancellation between the covariance error and the non-Gaussian error. At the level of the Gaussian expansion, the covariance perturbation is a second-order Hermite contribution, whereas the leading non-Gaussian correction lies in the third and fourth Hermite levels. A finite-dimensional separation argument for these components yields the additive scale
	\[
	M_{\C}(F_n)+\|\Sigma_n-\Sigma\|_{\HS}.
	\]
	
	Third, the quantity $M_{\C}$ is genuinely intrinsic to the complex chaos. The real theorem of \cite{CCL2026}, recalled as Theorem~\ref{thm:CCL-real-optimal}, is expressed through third-order joint cumulants and coordinate fourth cumulants of a real vector. For $(\mathrm{Re}\,F,\mathrm{Im}\,F)$, identifying this real quantity with \eqref{eq:intro-MC} requires nontrivial algebraic relations among complex moments. The further equivalence between $M_{\C}(F)$ and the contraction scale $\mathcal R_{p,q}(f)$ in \eqref{eq:contraction-rate} uses the complex product formula, conjugation, and the bidegree structure of $I_{p,q}(f)$. This is the step that converts a sharp real-vector theorem into a sharp theorem formulated entirely in the natural language of complex Wiener chaos.
	
	The application in Section~\ref{sec:app} shows that the two terms in the fixed-target rate are genuinely different. The complex Ornstein--Uhlenbeck process was introduced by Arat{\'o}, Kolmogorov and Sinai \cite{arato1962evaluation} in connection with the Chandler wobble. For the normalized numerator in the least-squares estimator of a complex fractional Ornstein--Uhlenbeck process, the intrinsic non-Gaussian term and the covariance mismatch can have different orders and different critical regimes. Consequently, the covariance term is not merely a technical correction. It determines the optimal rate to the limiting Gaussian law in part of the parameter range.

	At the density level, negative moments of Malliavin derivatives or determinants have long been used to obtain integration-by-parts formulas and derivative estimates. In the Breuer--Major setting, Hu, Nualart, Tindel and Xu \cite{HNTX2015} established the required negative moments under suitable assumptions and deduced uniform convergence of the densities and of all their derivatives. As emphasized by Herry, Malicet and Poly \cite{HMP2024}, this already constitutes a superconvergence result for properly normalized Hermite sums of stationary Gaussian sequences. Herry, Malicet and Poly subsequently proved a general superconvergence theorem on Wiener chaoses: convergence to a non-degenerate Gaussian vector automatically supplies negative moments of every prescribed finite order for the Malliavin determinant along a sufficiently far tail. These real regularization tools, combined with the intrinsic complex moment and contraction structure described above, lead to the complete optimal fixed-target statement of Theorem~\ref{thm:intro-main}.

	The paper is organized as follows. Section~\ref{sec:Preliminaries} recalls complex multiple Wiener--It\^o integrals, their associated real representation, and the Malliavin calculus used below. Section~\ref{sec:main-result} first develops the absolute-continuity, density-representation and Sobolev-regularity results, then proves the optimal Gaussian approximation bound. Section~\ref{sec:app} applies these results to a complex fractional Ornstein--Uhlenbeck process.

	\section{Preliminaries}\label{sec:Preliminaries}
	
	We introduce the framework of complex isonormal Gaussian processes
	and complex Wiener--It\^o integrals, describe their relation to the
	associated real isonormal Gaussian processes and real Wiener--It\^o
	integrals, and recall standard facts from Malliavin calculus. See \cite{ccl25,chen2017fourth,chenliu2019,ito1951real,ito1952complex,nourdin2012normal,nualart2006malliavin} for more details.
	
	\subsection{Complex isonormal Gaussian processes}

	Let $\mathfrak H$ be a complex separable Hilbert space with conjugation $\overline{\,\cdot\,}$, inner product $\langle\cdot,\cdot\rangle_{\mathfrak H}$ and norm $\|\cdot\|_{\mathfrak H}$. Let $(\Omega,\mathcal F,\mathbb P)$ be a complete probability space. A complex isonormal Gaussian process over $\mathfrak H$ is a centered complex Gaussian family $Z=\{Z(h):h\in\mathfrak H\}$ in $L^2(\Omega;\C)$, complex-linear in $h$, such that 
	\begin{equation}\label{eq:complex-isonormal}
		\E[Z(h)Z(g)]=0,
		\qquad
		\E[Z(h)\overline{Z(g)}]=\langle h,g\rangle_{\mathfrak H},
		\qquad h,g\in\mathfrak H.
	\end{equation}
	Here, we use the same notation $\overline{\,\cdot\,}$ to denote the complex conjugation for any complex element.

	For $p,q\in \N_0:=\N\cup\{0\}$ and $\rho>0$, the complex Hermite polynomials $J_{p,q}(\cdot,\rho)$, also called Hermite--Laguerre--It\^o polynomials, are defined by
	\begin{equation*}
		\exp\bigl\{\lambda\overline z+\overline\lambda z-\rho|\lambda|^2\bigr\}
		=
		\sum_{p,q\ge0}\frac{\overline\lambda^{\,p}\lambda^q}{p!q!}J_{p,q}(z,\rho),
		\qquad z,\lambda\in\C.
	\end{equation*}
	We write $a\wedge b:=\min\{a,b\}$ for $a,b\in\R$. Equivalently,
	\begin{equation*}
		J_{p,q}(z,\rho)
		=
		\sum_{r=0}^{p\wedge q}(-1)^r r!\binom pr\binom qr
		z^{p-r}\overline z^{\,q-r}\rho^r.
	\end{equation*}
	When $\rho=2$, we simply write $J_{p,q}(z)=J_{p,q}(z,2)$. For $p,q\in\N_0$, the $(p,q)$-th complex Wiener--It\^o chaos
	associated with $Z$, denoted by $\mathscr H_{p,q}(Z)$, is the closed
	linear subspace of $L^2(\Omega;\C)$ generated by random variables
	\[
	\left\{
	J_{p,q}(Z(h)):
	h\in\mathfrak H,\ 
	\|h\|_{\mathfrak H}=\sqrt2
	\right\}.
	\]
	Let $\sigma(Z):=\sigma\{Z(h):h\in\mathfrak H\}$
	be the $\sigma$-field generated by $Z$. The complex Wiener--It\^o chaos decomposition implies that the spaces $\mathscr H_{p,q}(Z)$ are mutually orthogonal and
	\begin{equation*}
		L^2
		(\Omega,\sigma(Z),\mathbb P; \C)
		=
		\bigoplus_{p=0}^{\infty}
		\bigoplus_{q=0}^{\infty}
		\mathscr H_{p,q}(Z).
	\end{equation*}

	We use the convention $\mathfrak H^{\otimes0}=\mathfrak H^{\odot0}=\C$ and $h^{\otimes0}=1$ for any $h\in\mathfrak H$. For $r\in\N$, $\mathfrak H^{\otimes r}$ denotes the $r$-th Hilbert tensor product of $\mathfrak H$, while $\mathfrak H^{\odot r}$ denotes its symmetric subspace. Let $S_r$ be the permutation group of $\{1,\ldots,r\}$. We define
	\[
	\operatorname{symm}(h_1\otimes\cdots\otimes h_r)
	:=
	\frac1{r!}
	\sum_{\pi\in S_r}
	h_{\pi(1)}\otimes\cdots\otimes h_{\pi(r)}, \qquad h_1,\ldots,h_r\in \mathfrak H.
	\]
	For $r=0$, the operator $\operatorname{symm}$ is the identity on $\C$. Let
	\[
	\Lambda
	:=
	\left\{
	\mathbf p=(p_k)_{k\ge1}:p_k\in\N_0,\ 
	p_k=0\text{ for all but finitely many }k
	\right\}.
	\]
	For $\mathbf p\in\Lambda$, set
	\[
	|\mathbf p|:=\sum_{k\ge1}p_k,
	\qquad
	\mathbf p!:=\prod_{k\ge1}p_k!.
	\]
	Let $\{\xi_k\}_{k\ge1}$ be a complete orthonormal system in $\mathfrak H$.
	For $\mathbf p,\mathbf q\in\Lambda$, define the complex Fourier--Hermite polynomial
	\[
	\mathbf J_{\mathbf p,\mathbf q}
	:=
	\prod_{k\ge1}
	\frac{1}{\sqrt{2^{p_k+q_k}p_k!q_k!}}
	J_{p_k,q_k}\!\left(\sqrt2\,Z(\xi_k)\right).
	\]
	Then for fixed $p,q\in\N_0$, the family
	\[
	\left\{
	\mathbf J_{\mathbf p,\mathbf q}:
	|\mathbf p|=p,\ |\mathbf q|=q
	\right\}
	\]
	is a complete orthonormal system in $\mathscr H_{p,q}(Z)$. As a consequence, the linear mapping
	\begin{align*}
		&I_{p,q}\left(
		\operatorname{symm}\left(\otimes_{k\ge1}\xi_k^{\otimes p_k}\right)
		\otimes
		\operatorname{symm}\left(\otimes_{k\ge1}\overline{\xi_k}^{\otimes q_k}\right)
		\right)
		:=
		\sqrt{\mathbf p!\,\mathbf q!}\,
		\mathbf J_{\mathbf p,\mathbf q}
	\end{align*}
	extends uniquely to an isometry from
	$\mathfrak H^{\odot p}\otimes\mathfrak H^{\odot q}$, equipped with the norm
	$\sqrt{p!q!}\|\cdot\|_{\mathfrak H^{\otimes(p+q)}}$, onto
	$\mathscr H_{p,q}(Z)$. For
	$f\in\mathfrak H^{\odot p}\otimes\mathfrak H^{\odot q}$,
	$I_{p,q}(f)$ is called the complex $(p,q)$-th Wiener--It\^o integral of $f$ with respect to $Z$.
	
	Given $f\in\mathfrak{H}^{\odot p}\otimes\mathfrak{H}^{\odot q}$ and $g\in\mathfrak{H}^{\odot r}\otimes\mathfrak{H}^{\odot s}$, where $p,q,r,s\in\N_0$, for $i=0,\ldots,p\wedge s$, $j=0,\ldots,q\wedge r$, the $(i,j)$-th contraction of $f$ and $g$ is an element of $\mathfrak{H}^{\otimes (p+r-i-j)}\otimes\mathfrak{H}^{\otimes (q+s-i-j)}$ defined as
	\begin{align*}
		f \otimes_{i, j} g
		&= \sum_{l_{1}, \ldots, l_{i+j}=1}^{\infty}\left\langle f, \xi_{l_{1}} \otimes \cdots \otimes \xi_{l_{i}} \otimes \overline{\xi_{l_{i+1}}} \otimes \cdots \otimes \overline{\xi_{l_{i+j}}}\right\rangle_{\mathfrak{H}^{\otimes (i+j)}}\\&\qquad\qquad\qquad\qquad\otimes\left\langle g, \xi_{l_{i+1}} \otimes \cdots \otimes \xi_{l_{i+j}} \otimes \overline{\xi_{l_{1}}} \otimes \cdots \otimes \overline{\xi_{l_{i}}}\right\rangle_{\mathfrak{H}^{\otimes (i+j)}},
	\end{align*}
	and by convention, $f \otimes_{0,0} g=f \otimes g$ denotes the tensor product of $f$ and $g$. In the preceding display, the brackets denote the partial inner products over the contracted variables. Let $f \widetilde{\otimes_{i, j}} g$ be the symmetrization of $f \otimes_{i, j} g$ in the sense of \cite[Eq. (5.1)]{ito1952complex}. The product formula for complex multiple Wiener--It\^o integrals, as presented in \cite[Theorem~2.1]{Chen2017} or \cite[Theorem~A.1]{Hoshino2017}, states that 
	\begin{equation}\label{eq:complex-product-formula}
		I_{p, q}(f) I_{r, s}(g)=\sum_{i=0}^{p \wedge s} \sum_{j=0}^{q \wedge r}\binom{p}{i} \binom{s}{i}\binom{q}{j}\binom{r}{j} i! j! I_{p+r-i-j, q+s-i-j}\left(f \widetilde{\otimes_{i, j}} g\right).
	\end{equation}
	For $p,q\in\N_0$, define first the conjugate-linear map
	\[
	\star:\mathfrak H^{\otimes p}\otimes\mathfrak H^{\otimes q}
	\rightarrow
	\mathfrak H^{\otimes q}\otimes\mathfrak H^{\otimes p}
	\]
	on elementary tensors by
	\begin{equation*}
		(h_1\otimes\cdots\otimes h_p
		\otimes
		g_1\otimes\cdots\otimes g_q)^\star
		:=
		\overline{g_1}\otimes\cdots\otimes\overline{g_q}
		\otimes
		\overline{h_1}\otimes\cdots\otimes\overline{h_p},
	\end{equation*}
	where $h_1,\ldots,h_p,g_1,\ldots,g_q\in\mathfrak H$, and extend it by conjugate linearity and continuity. This map sends $\mathfrak H^{\odot p}\otimes\mathfrak H^{\odot q}$ into $\mathfrak H^{\odot q}\otimes\mathfrak H^{\odot p}$. For $f\in\mathfrak H^{\odot p}\otimes\mathfrak H^{\odot q}$, $f^\star$ is called its reversed complex conjugate. According to \cite[Eq. (2.15)]{chenliu2019},
	\begin{equation}\label{eq:complex-integral-conjugate}
		\overline{I_{p,q}(f)}
		=
		I_{q,p}(f^\star).
	\end{equation}
	Here and below, $\mathbf 1_A$ denotes the indicator of an event or set $A$. The isometry property of complex multiple Wiener--It\^o integrals is that for $f\in
	\mathfrak H^{\odot p}\otimes\mathfrak H^{\odot q}$ and $g\in
	\mathfrak H^{\odot s}\otimes\mathfrak H^{\odot r}$, where $p,q,s,r\in \N_0$,
	one has
	\begin{equation}\label{eq:complex-isometry}
		\E\left[
		I_{p,q}(f)
		\overline{I_{s,r}(g)}
		\right]
		=
		\mathbf 1_{\{p=s,\ q=r\}}\,
		p!q!\,
		\langle f,g\rangle_{\mathfrak H^{\otimes(p+q)}}.
	\end{equation}

	\subsection{Associated real isonormal Gaussian processes}\label{sec:associated-real-Gaussian-processes}
	
	Let $\mathfrak H_{\R}:=\{h\in\mathfrak H:\overline h=h\}$. For $h,g\in\mathfrak H_{\R}$,
	\[
	\langle h,g\rangle_{\mathfrak H}
	=
	\overline{\langle h,g\rangle_{\mathfrak H}},
	\]
	so the restriction of $\langle\cdot,\cdot\rangle_{\mathfrak H}$, denoted by $\langle \cdot,\cdot\rangle_{\mathfrak H_{\R}}$, is real-valued and makes $\mathfrak H_{\R}$ a real separable Hilbert space. Every $h\in\mathfrak H$ has the unique decomposition
	\begin{equation*}
		h=h_1+\ii h_2,
		\qquad
		h_1=\frac{h+\overline h}{2}\in\mathfrak H_{\R},
		\qquad
		h_2=\frac{h-\overline h}{2\ii}\in\mathfrak H_{\R}.
	\end{equation*}
	Hence $\mathfrak H=\mathfrak H_{\R}\oplus \ii\mathfrak H_{\R}$.
	For $h\in\mathfrak H_{\R}$, define
	\begin{equation*}
		X(h):=\sqrt2\,\mathrm{Re}\, Z(h),
		\qquad
		Y(h):=\sqrt2\,\mathrm{Im}\, Z(h).
	\end{equation*}
	Using \eqref{eq:complex-isonormal}, we obtain, for
	$h,g\in\mathfrak H_{\R}$,
	\begin{equation*}
		\E[X(h)X(g)]=\E[Y(h)Y(g)]=\langle h,g\rangle_{\mathfrak H_{\R}},
		\qquad
		\E[X(h)Y(g)]=0.
	\end{equation*}
	Since the whole family is Gaussian, $X$ and $Y$ are independent real isonormal Gaussian processes over $\mathfrak H_{\R}$. Let $X_{\C}$ and $Y_{\C}$ denote their complex-linear
	extensions to $\mathfrak H$. That is, if
	$h=h_1+\ii h_2\in \mathfrak{H}$, with
	$h_1,h_2\in\mathfrak H_{\R}$, then
	\[
	X_{\C}(h)
	:=
	X(h_1)+\ii X(h_2),
	\qquad
	Y_{\C}(h)
	:=
	Y(h_1)+\ii Y(h_2).
	\]
	The original complex isonormal process $Z$ can be recovered as
	\begin{equation*}
		Z(h)
		=
		\frac{
			X_{\C}(h)
			+
			\ii Y_{\C}(h)
		}{\sqrt2},
		\qquad
		h\in\mathfrak H.
	\end{equation*}

	We now introduce the real Hilbert space on which the real isonormal
	Gaussian process associated with $Z$ will be defined. Set
	\begin{equation*}
		\mathfrak K
		:=
		\mathfrak H_{\R}
		\oplus
		\mathfrak H_{\R}
		=
		\left\{
		(h_1,h_2):
		h_1,h_2\in\mathfrak H_{\R}
		\right\}.
	\end{equation*}
	Then $\mathfrak K$ is a real separable Hilbert space equipped with the real inner product
	\begin{equation*}
		\left\langle
		(h_1,h_2),(g_1,g_2)
		\right\rangle_{\mathfrak K}
		:=
		\langle h_1,g_1\rangle_{\mathfrak H_{\R}}
		+
		\langle h_2,g_2\rangle_{\mathfrak H_{\R}},\qquad (h_1,h_2),(g_1,g_2) \in \mathfrak K.
	\end{equation*}
	Define
	\begin{equation*}
		W(h_1,h_2)
		:=
		X(h_1)+Y(h_2),
		\qquad
		(h_1,h_2)\in\mathfrak K.
	\end{equation*}
	Then $W=\{W(k):k\in\mathfrak K\}$ is a real isonormal Gaussian process
	over $\mathfrak K$. Indeed, for
	$k=(h_1,h_2)$ and $\ell=(g_1,g_2)$ in $\mathfrak K$,
	the independence of $X$ and $Y$ yields
	\begin{align*}
		\E[W(k)W(\ell)]
		&=
		\E[X(h_1)X(g_1)]
		+
		\E[Y(h_2)Y(g_2)]
		\\
		&=
		\langle h_1,g_1\rangle_{\mathfrak H_{\R}}
		+
		\langle h_2,g_2\rangle_{\mathfrak H_{\R}}
		=
		\langle k,\ell\rangle_{\mathfrak K}.
	\end{align*}
	
	For $r\in\N_0$, let $H_r(\cdot)$ be the real Hermite polynomial of order $r$ determined by the equality
	\begin{equation*}
		\exp\left\{tx-\frac{t^2}{2}\right\}
		=
		\sum_{r=0}^{\infty}\frac{t^r}{r!}H_r(x), \qquad x,t\in\R.
	\end{equation*}
	For $r\in\N_0$, the $r$-th real Wiener--It\^o chaos
	associated with $W$, denoted by $\mathcal H_r(W)$, is defined as
	the closed real linear subspace of $L^2(\Omega;\R)$
	generated by random variables
	\[
	\left\{
	H_r(W(k)):
	k\in\mathfrak K,\ 
	\|k\|_{\mathfrak K}=1
	\right\}.
	\]
	Here ``linear'' means linear over $\R$. Equivalently,
	\[
	\mathcal H_r(W)
	=
	\overline{
		\operatorname{span}_{\R}
		\left\{
		H_r(W(k)):
		k\in\mathfrak K,\ 
		\|k\|_{\mathfrak K}=1
		\right\}
	}^{\,L^2(\Omega;\R)}.
	\]
	Let $\sigma(W):=\sigma\{W(k):k\in\mathfrak K\}$. The Wiener--It\^o chaos decomposition implies that the spaces $\mathcal{H}_{r}(W)$ are mutually orthogonal and
	\begin{equation*}
		L^2(\Omega,\sigma(W),\mathbb P;\R)
		=\bigoplus_{r=0}^{\infty} \mathcal{H}_{r}(W).
	\end{equation*}

	For $r\in \N$, let $\mathfrak K^{\otimes r}$ and $\mathfrak K^{\odot r}$ denote the $r$-th Hilbert tensor product and symmetric tensor product of $\mathfrak K$, respectively, and set $\mathfrak K^{\otimes 0}= \mathfrak K^{\odot0}=\R$. For any $r\in \N$, the mapping $I_r^W(k^{\otimes r}):=H_r(W(k))$ for $k\in\mathfrak K$ with $\|k\|_{\mathfrak K}=1$ extends uniquely to a linear isometry between $\mathfrak K^{\odot r}$, equipped with the norm $\sqrt{r!}\|\cdot\|_{\mathfrak K^{\otimes r}}$, and the $r$-th Wiener--It\^o chaos $\mathcal H_r(W)$. For $r=0$, we write $I_0^W(c)=c$ for $c\in\R$. For any $u\in\mathfrak K^{\odot r}$, the random variable $I_r^W(u)$ is called the $r$-th Wiener--It\^o integral of $u$ with respect to $W$. Wiener--It\^o chaoses of different orders are orthogonal, that is, for $u\in\mathfrak K^{\odot r}$ and $v\in\mathfrak K^{\odot s}$, where $r,s\in\N_0$, 
	\begin{equation*}
		\E[I_r^W(u)I_s^W(v)]
		=
		\mathbf 1_{\{r=s\}}r!\langle u,v\rangle_{\mathfrak K^{\otimes r}}.
	\end{equation*}

	The following result is from \cite[Theorem~3.3]{chen2017fourth} and provides a bridge between the complex and real theories; see also \cite[Theorem~2.19]{chenliu2019}.
	
	\begin{proposition}\label{prop:complex-real}
		Let $F=I_{p,q}(f)$, where $p,q\in \N_0$ and $f\in\mathfrak H^{\odot p}\otimes\mathfrak H^{\odot q}$.
		Then there exist real symmetric kernels $u,v\in\mathfrak K^{\odot(p+q)}$ such that
		\begin{equation*}
			\mathrm{Re}\, F=I_{p+q}^W(u),
			\qquad
			\mathrm{Im}\, F=I_{p+q}^W(v).
		\end{equation*}
	\end{proposition}
	
	The explicit formulas for the kernels of the real and imaginary parts are provided
	in \cite{ccl22}. Thus the real and imaginary parts of a complex integral of bidegree $(p,q)$ always belong to the same real Wiener--It\^o chaos of order $p+q$. This fact will be used repeatedly below.
	
	We also record the multivariate real optimal bound in \cite[Theorem~1.1, Corollary~1.2]{CCL2026}, which will later be transferred to the complex setting. For $k\in\N_0$ and $d\in\N$, set $[d]^k:=\{1,\ldots,d\}^k$, with $[d]^0=\{\emptyset\}$ and $[d]:=[d]^1$. An element $\alpha=(\alpha_1,\ldots,\alpha_k)\in[d]^k$ is called an ordered multi-index of length $|\alpha|:=k$, and we write $\partial_\alpha=\partial_{\alpha_1}\cdots\partial_{\alpha_k}$ and $\partial_\emptyset=\operatorname{Id}$. If $\beta\in[d]^\ell$, then $(\beta,\alpha)$ denotes the ordered multi-index of length $\ell+k$. For $m\in\N_0$ and $r\in[1,\infty]$, let
	$W^{m,r}(\R^d)$ denote the usual Sobolev space consisting of all
	$f\in L^r(\R^d)$ such that $\partial_\alpha f\in L^r(\R^d)$ for every
	$\alpha\in[d]^k$ with $1\le k\le m$,
	where the partial derivatives are understood in the distributional sense.
	We equip $W^{m,r}(\R^d)$ with the norm
	\begin{equation}\label{eq:sobolev-norm}
		\|f\|_{W^{m,r}(\R^d)}:=\sum_{k=0}^m\sum_{\alpha\in[d]^k}\|\partial_\alpha f\|_{L^r(\R^d)}.
	\end{equation}
	For $r=\infty$, the $L^\infty$ norms are understood as essential suprema. For $\R^d$-valued random vectors $X,Y$ (with finite first moments when the $1$-Wasserstein distance is considered), we write the total variation, Kolmogorov and $1$-Wasserstein distances as
	\begin{equation}\label{eq:probability-distances}
		\begin{aligned}
			d_{\mathrm{TV}}(X,Y)&:=\sup_{A\in\mathcal B(\R^d)}|\mathbb P(X\in A)-\mathbb P(Y\in A)|,\\
			d_{\mathrm{Kol}}(X,Y)&:=\sup_{x\in\R^d}|\mathbb P(X_1\le x_1,\ldots,X_d\le x_d)-\mathbb P(Y_1\le x_1,\ldots,Y_d\le x_d)|,\\
			d_{\rm W}(X,Y)&:=\sup_{\operatorname{Lip}(h)\le1}|\E[h(X)]-\E[h(Y)]|,
		\end{aligned}
	\end{equation}
	where $\mathcal B(\R^d)$ denotes the Borel $\sigma$-field on $\R^d$ and the last supremum is over all real-valued Lipschitz functions on $\R^d$. For $\alpha=(\alpha_1,\ldots,\alpha_k)\in[d]^k$ with $k\geq1$ and a centered random vector $\mathbf G=(G_1,\ldots,G_d)$ satisfying $\E[|\mathbf G|^k]<\infty$, the joint cumulant of $\mathbf G$ indexed by $\alpha$ is
	\begin{equation*}
		\kappa_\alpha(\mathbf G)
		:=\left.(-\ii)^k\partial_\alpha\log\varphi_{\mathbf G}(t)\right|_{t=0},
	\end{equation*}
	where $\varphi_{\mathbf G}(t):=\E\big[\exp\bigl(\ii\langle t,\mathbf G \rangle_{\R^d}\bigr)\big],
	\, t\in\R^d$, is the characteristic function of $\mathbf G$. For a centered real random variable $X$ with finite fourth moment, write
	\[
	\kappa_4(X):=\E[X^4]-3\bigl(\E[X^2]\bigr)^2
	\]
	for its fourth cumulant. For $i\in[d]$, one has $\kappa_{(i,i,i,i)}(\mathbf G)=\kappa_{4}(G_i)$. Set
	\begin{equation}\label{eq:Md-real-prelim}
		M_d(\mathbf G)
		:=
		\max\left\{
		\max_{\alpha\in[d]^3}|\kappa_\alpha(\mathbf G)|,
		\max_{1\le i\le d}\kappa_4(G_i)
		\right\}.
	\end{equation}
	For $d=2$, we use the shorthand $M_2(F_1,F_2):=M_2((F_1,F_2))$. In the fixed-chaos setting considered below, each coordinate fourth cumulant in \eqref{eq:Md-real-prelim} is non-negative.

	\begin{theorem}[{\cite[Theorem~1.1 and Corollary~1.2]{CCL2026}}]\label{thm:CCL-real-optimal}
		Fix $d\ge1$ and $k\ge2$. For $n\ge1$, let
		\[
		\mathbf G_n=(G_{n,1},\ldots,G_{n,d}),
		\qquad
		G_{n,i}=I_k^W(g_{n,i}),\qquad g_{n,i}\in\mathfrak K^{\odot k},
		\]
		and assume
		\[
		\Cov(\mathbf G_n)=I_d,
		\qquad
		\mathbf G_n\xrightarrow{\mathrm{law}}N_d(0,I_d).
		\]
		Then, for every $m\in\N_0$ and every $r\in[1,\infty]$, $\mathbf G_n$ has a density $p_{\mathbf G_n}\in W^{m,r}(\R^d)$ for all sufficiently large $n$ and 
		\begin{equation*}
			\|p_{\mathbf G_n}-\phi_d\|_{W^{m,r}(\R^d)}
			\asymp 	d_\ast\bigl(\mathbf G_n,N_d(0,I_d)\bigr)\asymp
			M_d(\mathbf G_n),
		\end{equation*}
		where $d_\ast\in\{d_{\mathrm{TV}},d_{\mathrm{Kol}},d_{\rm W}\}$ and  $\phi_d(x)=(2\pi)^{-d/2}e^{-|x|^2/2}$ is the standard Gaussian density on $\R^d$.
	\end{theorem}
	
	This theorem is the real quantitative input for Section~\ref{sec:optimal-gaussian-approximation}. Proposition~\ref{prop:complex-real} places $(\mathrm{Re}\,F,\mathrm{Im}\,F)$ in a fixed real chaos, while Lemmas~\ref{lem:MC-real} and~\ref{lem:MC-contraction} identify the resulting real rate with intrinsic complex moments and contractions.
	
	\subsection{Malliavin calculus on the associated real Gaussian space}\label{sec:malliavin-calculus}
	
	We first recall the real Malliavin calculus in an abstract form. Let $\mathfrak U$ be a real separable Hilbert space with inner product $\left\langle \cdot,\cdot\right\rangle _\mathfrak{U}$, and $G=\{G(h):h\in\mathfrak U\} $
	be a real isonormal Gaussian process over $\mathfrak U$. That is, $G$ is a centered Gaussian family in $L^2(\Omega; \R)$ such that $\E\left[G(h)G(g)\right]=\left\langle h,g\right\rangle _\mathfrak{U}$ for any $h,g\in \mathfrak{U}$.

	For $k\in\N_0$ and $d\in\N$, let $C_p^k(\R^d)$ and
	$C_b^k(\R^d)$ denote, respectively, the class of $k$-times continuously differentiable functions on $\R^d$ whose partial derivatives of order at most $k$ have at most polynomial growth and the class for which those derivatives are bounded. The spaces $C_p^\infty(\R^d)$ and $C_b^\infty(\R^d)$ have the corresponding meanings at every order, while $C_c^\infty(\R^d)$ denotes the space of infinitely differentiable functions with compact support.
	
	Let $\mathcal S_G$ be the class of smooth cylindrical random variables
	\[
	F=f(G(h_1),\ldots,G(h_d)),
	\]
	where $h_1,\ldots,h_d\in\mathfrak U$ and $f\in C_p^\infty(\R^d)$. Given $F\in \mathcal S_G$, the Malliavin derivative is a $\mathfrak U$-valued random element defined as 
	\begin{equation*}
		DF
		=
		\sum_{j=1}^d\partial_jf(G(h_1),\ldots,G(h_d))h_j.
	\end{equation*}
	The derivative operator $D$ is closable from $L^s(\Omega)$ to $L^s(\Omega;\mathfrak U)$ for every $s\ge 1$. By iteration one obtains the $k$-th derivative $D^kF\in L^s(\Omega;\mathfrak U^{\otimes k})$. For $k\in\N_0$ and $s\ge 1$, we denote the Malliavin--Sobolev space $\mathbb D^{k,s}$ as the completion of $\mathcal S_G$ under
	\begin{equation*}
		\|F\|_{k,s}
		:=
		\left(\sum_{j=0}^k\E\|D^jF\|_{\mathfrak U^{\otimes j}}^s\right)^{1/s}.
	\end{equation*}
	We omit the subscript $k$ when $k=0$, namely $\|F\|_s:=(\E|F|^s)^{1/s}$. We also write $\mathbb D^\infty=\bigcap_{k\ge0}\bigcap_{s>1}\mathbb D^{k,s}$. For a separable Hilbert space $V$, we use
	$\mathbb D^{k,s}(V)$ and $\|\cdot\|_{k,s;V}$ for the corresponding
	$V$-valued Malliavin--Sobolev space and norm. We suppress $V$ from the notation when no confusion can
	arise. The derivative operator $D$ satisfies the chain rule. Specifically, if $\varphi\in C_b^1(\R^d)$ and $\mathbf F=\left(F_1, \ldots, F_d\right)\in (\mathbb{D}^{1,2})^d$, then $\varphi(\mathbf F) \in \mathbb{D}^{1,2}$ and
	\begin{equation*}
		D \varphi(\mathbf F)=\sum_{i=1}^d \frac{\partial \varphi}{\partial x_i}(\mathbf F)DF_i.
	\end{equation*}
	The chain rule still holds if $\mathbf F \in (\mathbb{D}^{\infty})^d$ and $\varphi\in C_p^1(\R^d)$. For a random vector $\mathbf F=(F_1,\ldots,F_d)\in(\mathbb D^{1,2})^d$, its Malliavin matrix is
	\begin{equation*}
		\gamma_{\mathbf F}
		=
		\bigl(\langle DF_i,DF_j\rangle_{\mathfrak U}\bigr)_{1\le i,j\le d}.
	\end{equation*}
	In particular, $\det\gamma_{\mathbf F}>0$ almost surely implies absolute continuity of the law of $\mathbf F$.
	
	The divergence operator $\delta$ is the adjoint of $D$. A random element $u\in L^2(\Omega;\mathfrak U)$ belongs to the domain of $\delta$, denoted by $\Dom(\delta)$, if and only if there exists a constant
	$c_u<\infty$ depending only on $u$, such that 
	\[
	|\E\langle DF,u\rangle_{\mathfrak U}|\le c_u\|F\|_2 ,\qquad  \forall\, F \in \mathbb{D}^{1,2}.
	\]
	In particular, if $u\in\Dom(\delta)$, then $\delta(u)$ is characterized by the following duality relationship
	\begin{equation*}
		\E[F\delta(u)]
		=
		\E\langle DF,u\rangle_{\mathfrak U}.
	\end{equation*}
	For every
	$k\in\N_0$ and $s>1$, the divergence operator $\delta$ is continuous from
	$\mathbb D^{k+1,s}(\mathfrak U)$ into $\mathbb D^{k,s}$; namely, there is $c_{k,s}<\infty$ such that
	\begin{equation*}
		\|\delta(u)\|_{k,s}
		\le c_{k,s}\|u\|_{k+1,s;\mathfrak U}.
	\end{equation*}

	For later use we also record the regularity of fixed Wiener--It\^o chaoses. Let $I_q^G$, $q\in\N_0$, be the real multiple Wiener--It\^o integral associated with $G$. If $q\ge1$, $f\in\mathfrak U^{\odot q}$, then $I_q^G(f)\in\mathbb D^\infty$ and its Malliavin derivative is 
	\begin{equation*}
		D I_q^G(f)=qI_{q-1}^G(f).
	\end{equation*}
	By Meyer's inequalities \cite[Theorem~1.5.1]{nualart2006malliavin},
	and hypercontractivity \cite[Corollary~2.8.14]{nourdin2012normal}, the following standard estimate holds: for fixed $q,k\in \N_0$ and $s\ge1$,
	\begin{equation}\label{eq:fixed-chaos-Sob}
		\|I_q^G(f)\|_{k,s}
		\le c_{q,k,s}\|I_q^G(f)\|_2.
	\end{equation}
	
	From now on all Malliavin operators are taken with respect to the real isonormal process $W$ over $\mathfrak K$ introduced in Subsection~\ref{sec:associated-real-Gaussian-processes}. In particular, if
	\[
	F=I_{p,q}(f)=F_1+\ii F_2, \qquad f\in \mathfrak{H}^{\odot p}\otimes \mathfrak{H}^{\odot q}, \qquad p,q\in \N_0,
	\]
	then the Malliavin matrix associated with $F$ means the Malliavin matrix of the real vector $\widetilde F=(F_1,F_2)$:
	\begin{equation*}
		\gamma_F=\gamma_{\widetilde F}
		=
		\begin{pmatrix}
			\|DF_1\|_{\mathfrak K}^2 & \langle DF_1,DF_2\rangle_{\mathfrak K}\\
			\langle DF_1,DF_2\rangle_{\mathfrak K} & \|DF_2\|_{\mathfrak K}^2
		\end{pmatrix}.
	\end{equation*}
	
	\section{Density regularity and optimal Gaussian approximation}\label{sec:main-result}
	
	\subsection{Absolute continuity, density representations and Sobolev regularity}\label{sec:density-regularity}
	
	We first give the absolute-continuity criterion for a complex multiple Wiener--It\^o integral. This statement requires no negative-moment assumption on the Malliavin determinant.
	
	If $U=I_k^W(u)$ and $V=I_k^W(v)$ are two real multiple Wiener--It\^o integrals of the same order $k\ge1$, Nualart and Tudor \cite[Theorem~3]{NT2017} proved that the law of $(U,V)$ is absolutely continuous with respect to Lebesgue measure on $\R^2$ if and only if its covariance matrix is nonsingular. Equivalently, failure of absolute continuity is characterized by $\det\Cov(U,V)=0$.
	
	\begin{theorem}\label{thm:density-existence}
		Let $F=I_{p,q}(f)=F_1+\ii F_2$, where $p,q\in\N_0$, $p+q\ge1$ and $f\in\mathfrak H^{\odot p}\otimes\mathfrak H^{\odot q}$.
		Then the following hold.
		\begin{enumerate}[(i)]
			\item If $p\ne q$, the law of $(F_1,F_2)$ is absolutely continuous with respect to Lebesgue measure on $\R^2$ if and only if $f\ne0$.
			\item If $p=q$, the law of $(F_1,F_2)$ fails to be absolutely continuous if and only if there exists $\theta\in[0,2\pi)$ such that
			\begin{equation}\label{eq:density-singular-kernel}
				f=e^{\ii\theta}f^\star.
			\end{equation}
		\end{enumerate}
	\end{theorem}
	
	\begin{proof}
		By Proposition~\ref{prop:complex-real}, $F_1$ and $F_2$ are real multiple Wiener--It\^o integrals of the same order $p+q$ with respect to $W$. By \cite[Theorem~3]{NT2017} recalled above, it is therefore enough to characterize when $\Cov(\widetilde F)$ is singular, namely $\det \Cov(\widetilde F)=0$, where $\widetilde F=(F_1,F_2)$.
		
		Suppose first that $p\ne q$. The isometry property \eqref{eq:complex-isometry} gives
		\begin{align*}
			\E[F^2]&= \E\left[F \overline{\overline F}\right]= \E\left[I_{p,q}(f)  \overline{I_{q,p}(f^\star)}\right] =0,\\
			\E|F|^2&=\E\left[F\overline F\right]= \E\left[I_{p,q}(f)  \overline{I_{p,q}(f)}\right]= p!q!\|f\|_{\mathfrak{H}^{\otimes (p+q)}}^2.
		\end{align*}
		Writing $F=F_1+\ii F_2$, we obtain
		\[
		\E[F_1^2]=\E[F_2^2]=\frac12\E|F|^2 =\frac12 p!q!\|f\|_{\mathfrak{H}^{\otimes (p+q)}}^2,
		\qquad
		\E[F_1F_2]=0.
		\]
		Consequently, $\Cov(\widetilde F)=  p!q!\|f\|_{\mathfrak{H}^{\otimes (p+q)}}^2 I_2/2$ and $\det \Cov(\widetilde F)>0$ if and only if $f\ne0$.
		
		Now let $p=q$. Since $F_1$ and $F_2$ are centered, $\Cov(\widetilde F)$ is singular if and only if there exists $(a,b)\in\R^2\setminus\{0\}$ such that
		\[
		aF_1+bF_2=0\quad\text{a.s.}
		\]
		Using $F_1=(F+\overline F)/2$ and $F_2=(F-\overline F)/(2\ii)$,
		we rewrite this relation as
		\[
		(a-\ii b)F+(a+\ii b)\overline F=0\quad\text{a.s.}
		\]
		Choose $r>0$ and $\vartheta\in[0,2\pi)$ such that $a-\ii b=re^{-\ii\vartheta}$. Dividing by $r$ gives  
		\[
		e^{-\ii\vartheta}F+e^{\ii\vartheta}\overline F=0\quad\text{a.s.}
		\]
		By \eqref{eq:complex-integral-conjugate} and the injectivity of $I_{p,p}$ (which follows from the isometry property \eqref{eq:complex-isometry}),
		\[
		e^{-\ii\vartheta}f+e^{\ii\vartheta}f^\star=0.
		\]
		Equivalently, $f=-e^{2\ii\vartheta}f^\star$.
		As $-e^{2\ii\vartheta}$ ranges over the unit circle, this is precisely \eqref{eq:density-singular-kernel}. 
	\end{proof}
	
	The following corollary was previously obtained in \cite[Corollary 3.10]{ccl22} by means of the explicit kernel representation. Theorem~\ref{thm:density-existence} provides a shorter proof.
	
	\begin{corollary}\label{cor:basis-density}
		Fix $p,q\in\N_0$ and $p+q\ge1$. Let
		\[
		f=
		\operatorname{symm}(\xi_{k_1}\otimes\cdots\otimes \xi_{k_p})
		\otimes
		\operatorname{symm}(\overline{\xi_{j_1}}\otimes\cdots\otimes \overline{\xi_{j_q}}), \qquad 
		\]
		where $\{\xi_k\}_{k\ge1}$ is an orthonormal basis of $\mathfrak H$, $1\le k_1\le\cdots\le k_p$ and $1\le j_1\le\cdots\le j_q$. Then $I_{p,q}(f)$ has a density if and only if either $p\ne q$, or $p=q$ and there exists $1\leq l\leq p$ such that $k_l\neq j_l$.
	\end{corollary}
	
	\begin{proof}
		If $p\ne q$, the kernel is nonzero and Theorem~\ref{thm:density-existence}(i) applies. If $p=q$, the reversed conjugate of $f$ is
		\[
		f^\star=\operatorname{symm}(\xi_{j_1}\otimes\cdots\otimes \xi_{j_q})
		\otimes
		\operatorname{symm}(\overline{\xi_{k_1}}\otimes\cdots\otimes \overline{\xi_{k_p}}).
		\]
		
		Since these basis tensors are orthogonal unless their index multisets coincide, $f$ can be a unimodular multiple of $f^\star$ exactly when the two multisets coincide.
	\end{proof}

	Having characterized absolute continuity, we turn to finite-order regularity of the density. Since the Malliavin operators act on the associated real Gaussian space, we use the two-dimensional, unweighted specialization of the estimates in \cite[Section~4]{CCL2026}. In particular, only one sufficiently high finite negative moment of the Malliavin determinant is required.

	Let $\mathbf F=(F_1,F_2)$ with $F_i\in\mathbb D^{1,2}$ and Malliavin matrix $\gamma_{\mathbf F}$. For $r\in\N_0$ and $P,Q>1$, we say that $\mathbf F$ satisfies $\mathcal A(r,P,Q)$ if
	\begin{equation}\label{eq:A-condition}
		K_+(r,P):=\max_{i=1,2}\|F_i\|_{r,P}<\infty,
		\qquad
		K_-(Q):=\E[(\det\gamma_{\mathbf F})^{-Q}]<\infty.
	\end{equation}
	The second condition implies $\det\gamma_{\mathbf F}>0$ almost surely.  Constants below are allowed to depend on the corresponding values of $K_+(r,P)$ and $K_-(Q)$.

	For $i=1,2$, define the $\mathfrak K$-valued random field
	\begin{equation*}
		u_i(\mathbf F)
		:=
		\sum_{j=1}^2(\gamma_{\mathbf F}^{-1})_{ij}DF_j.
	\end{equation*}
	For $\alpha=(\alpha_1,\ldots,\alpha_\ell)\in[2]^\ell$, we define recursively
	\begin{equation*}
		H_{\emptyset}(\mathbf F):=1,
		\qquad
		H_\alpha(\mathbf F)
		:=
		\delta\left(H_{(\alpha_1,\ldots,\alpha_{\ell-1})}(\mathbf F)u_{\alpha_\ell}(\mathbf F)\right),
	\end{equation*}
	whenever the successive divergences are well defined.
	
	The next lemma shows that these weights $H_\alpha(\mathbf F)$ are well-defined under the finite-order assumptions stated there. This is the specialization of \cite[Proposition~4.4]{CCL2026} to $d=2$ and weight $G=1$.
	
	\begin{lemma}\label{lem:finite-IBP}
		Fix $\ell,r\in\N_0$ and $s>1$. There exist finite exponents $P=P(\ell,r,s)>1$ and $Q=Q(\ell,r,s)>1$
		such that, if $\mathbf F$ satisfies $\mathcal A(r+\ell+1,P,Q)$, then $H_\alpha(\mathbf F)$ is well defined for every $\alpha\in[2]^\ell$ and
		\begin{equation*}
			\|H_\alpha(\mathbf F)\|_{r,s}\le C,
		\end{equation*}
		where $C$ depends on $r,\ell,s$ and on $K_+(r+\ell+1,P)$, $K_-(Q)$ in \eqref{eq:A-condition}. Moreover, for every $\varphi\in C_b^\ell(\R^2)$,
		\begin{equation}\label{eq:IBP-H}
			\E[\partial_\alpha\varphi(\mathbf F)]
			=
			\E[\varphi(\mathbf F)H_\alpha(\mathbf F)].
		\end{equation}
	\end{lemma}

	We next derive explicit density formulas. An oriented representation is convenient because the orthant can be chosen according to the evaluation point $x$.
	
	\begin{theorem}\label{thm:finite-density-formula}
		Fix $m\in\N_0$. There exist finite exponents $P=P(m)>1$ and $Q=Q(m)>1$ such that, if $\mathbf F=(F_1,F_2)$ satisfies $\mathcal A(m+3,P,Q)$, then its law has a density $p_{\mathbf F}\in C^m(\R^2)$. Moreover, for every $\alpha\in[2]^k$ with $0\le k\le m$, every $A\subset\{1,2\}$, with $|A|$ denoting the cardinality of $A$, and every $x=(x_1,x_2)\in\R^2$,
		\begin{equation}\label{eq:oriented-density}
			\partial_\alpha p_{\mathbf F}(x)
			=
			(-1)^{k+|A|}\E\bigg[
			\bigg( \prod_{i\notin A}\mathbf 1_{\{F_i>x_i\}}\bigg) 
			\bigg( \prod_{i\in A}\mathbf 1_{\{F_i<x_i\}}\bigg) 
			H_{(1,2,\alpha)}(\mathbf F)
			\bigg].
		\end{equation}
	\end{theorem}
	
	\begin{proof}
		Choose $P,Q$ sufficiently large so that Lemma~\ref{lem:finite-IBP} applies to all ordered indices of length at most $m+2$. Let $h\in C_c^\infty(\R^2)$ and define
		\[
		\Phi_h(y_1,y_2)
		:=
		\int_{-\infty}^{y_1}\int_{-\infty}^{y_2}h(x_1,x_2)\,\mathrm dx_2\,\mathrm dx_1.
		\]
		Then $\Phi_h\in C_b^\infty(\R^2)$ and $\partial_1\partial_2\Phi_h=h$. Applying \eqref{eq:IBP-H} and then Fubini's theorem,
		\begin{align*}
			\E[h(\mathbf F)]
			&=\E[\Phi_h(\mathbf F)H_{(1,2)}(\mathbf F)]=\int_{\R^2}h(x)
			\E\left[\mathbf 1_{\{F_1>x_1\}}\mathbf 1_{\{F_2>x_2\}}H_{(1,2)}(\mathbf F)\right]dx.
		\end{align*}
		This proves absolute continuity and \eqref{eq:oriented-density} for $k=0$ and $A=\varnothing$. For $\alpha\in[2]^k$ with $1\le k\le m$, apply the same argument to $\partial_\alpha h$. In the sense of distributions,
		\begin{align*}
			\int_{\R^2}h(x)\partial_\alpha p_{\mathbf F}(x)\,\mathrm dx
			& =(-1)^k\E[\partial_\alpha h(\mathbf F)] ,
		\end{align*}
		and a further application of \eqref{eq:IBP-H} gives \eqref{eq:oriented-density} with $A=\varnothing$. For a general $A\subset\{1,2\}$, replace $\Phi_h$ by the oriented primitive
		\[
		\Phi_{h,A}(y)
		:=
		\int_{\prod_{i\notin A}(-\infty,y_i)\times\prod_{i\in A}(y_i,\infty)}h(x)\,\mathrm dx.
		\]
		Then $\partial_1\partial_2\Phi_{h,A}=(-1)^{|A|}h$, and the same argument gives \eqref{eq:oriented-density}.
		
		Finally, the case $k=0$ shows that $\mathbf F$ has a density, hence $\mathbb P(F_i=x_i)=0$ for every fixed $i$ and $x_i$. The weights $H_{(1,2,\alpha)}(\mathbf F)$ belong to $L^1$ by Lemma~\ref{lem:finite-IBP}. Therefore the right-hand side of \eqref{eq:oriented-density} is continuous in $x$ by the dominated convergence theorem. Thus $p_{\mathbf F}\in C^m(\R^2)$.
	\end{proof}
	
	\begin{corollary}\label{cor:Sob-density}
		Fix $m\in\N_0$ and $r\in[1,\infty]$. There exist finite exponents $P=P(m,r)>1$ and $Q=Q(m,r)>1$ such that, if $\mathbf F$ satisfies $\mathcal A(m+3,P,Q)$, then
		\begin{equation}\label{eq:Sob-density-bound}
			p_{\mathbf F}\in W^{m,r}(\R^2),
			\qquad
			\|p_{\mathbf F}\|_{W^{m,r}(\R^2)}\le C,
		\end{equation}
		where $C$ depends only on $m,r$, and on $K_+(m+3,P)$, $K_-(Q)$ in \eqref{eq:A-condition}.
	\end{corollary}
	
	\begin{proof}
		For $x\in\R^2$, set
		\[
		A_x:=\{i\in\{1,2\}:x_i<0\}, \qquad 	B_x:=
		\bigcap_{i\notin A_x}\{F_i>x_i\}
		\cap
		\bigcap_{i\in A_x}\{F_i<x_i\}.
		\]
		On $B_x$ one has $|x_i|\le|F_i|$ for $i=1,2$, hence $|x|\le|\mathbf F|$. Fix an integer $a\ge0$. Applying the oriented formula \eqref{eq:oriented-density}, H\"older's inequality, and Lemma~\ref{lem:finite-IBP}, for every $|\alpha|\le m$, after choosing the finite exponents $P,Q$ large enough,
		\begin{align*}
			(1+|x|)^a|\partial_\alpha p_{\mathbf F}(x)|
			&\le
			\E\left[(1+|\mathbf F|)^a|H_{(1,2,\alpha)}(\mathbf F)|\right]\le C.
		\end{align*}
		Consequently,
		\begin{equation*}
			|\partial_\alpha p_{\mathbf F}(x)|\le C(1+|x|)^{-a},
			\qquad |\alpha|\le m.
		\end{equation*}
		If $1\le r<\infty$, choose $a$ so that $ar>2$; then $\partial_\alpha p_{\mathbf F}\in L^r(\R^2)$. For $r=\infty$, the derivatives of order at most $m$ are understood as the continuous versions ensured by Theorem~\ref{thm:finite-density-formula}, and $\partial_\alpha p_{\mathbf F}(x)$ is bounded by taking $a=0$. Summing over $|\alpha|\le m$ proves \eqref{eq:Sob-density-bound}.
	\end{proof}
	
	\begin{corollary}\label{cor:complex-density-finite}
		Let $F=I_{p,q}(f)$ with $p,q\in \N_0$ and $p+q\ge1$. For each fixed $m\in\N_0$ and $r\in[1,\infty]$, there exists a finite $Q=Q(m,r)>1$ such that, if
		\[
		\E[(\det\gamma_F)^{-Q}]<\infty,
		\]
		then the representation \eqref{eq:oriented-density} holds for $\widetilde F=(\mathrm{Re}\,F,\mathrm{Im}\,F)$ for every ordered multi-index $\alpha$ with $|\alpha|\le m$, and $p_F:=p_{\widetilde F}\in W^{m,r}(\R^2)$.
	\end{corollary}
	
	\begin{proof}
		By Proposition~\ref{prop:complex-real}, both components of $\widetilde F$ belong to the fixed real Wiener--It\^o chaos $\mathcal H_{p+q}(W)$. Thus all positive Malliavin--Sobolev norms $K_+(m+3,P)$ required in $\mathcal A(m+3,P,Q)$ are finite by \eqref{eq:fixed-chaos-Sob}. The conclusion now follows from Theorem~\ref{thm:finite-density-formula} and Corollary~\ref{cor:Sob-density}.
	\end{proof}
	For sequences in a fixed Wiener chaos that converge to a non-degenerate Gaussian law, the finite negative moments required in Corollary~\ref{cor:complex-density-finite} hold along a sufficiently far tail by \cite[Theorem~4]{HMP2024}. This is the regularization input behind the real optimal theorem \cite{CCL2026}. 
	
	\subsection{Optimal Gaussian approximation}\label{sec:optimal-gaussian-approximation}

	We now identify the optimal approximation scale. We first compare the real quantity $M_2$ in \eqref{eq:Md-real-prelim} with the intrinsic complex quantity $M_{\C}$ in \eqref{eq:intro-MC}, and then with a scale expressed through complex kernel contractions. These identifications yield the covariance-matched estimate and, after controlling covariance perturbations, the fixed-target result. Proofs of the auxiliary algebraic and stability lemmas are collected in Subsection~\ref{subsec:auxiliary-proofs}.

	\begin{lemma}\label{lem:MC-real}
		Fix $p,q\in\N_0$ with $p+q\ge2$. There exist constants $0<c\le C<\infty$, depending only on $p+q$, such that for every $F=I_{p,q}(f)=F_1+\ii F_2$, where $f\in\mathfrak H^{\odot p}\otimes\mathfrak H^{\odot q}$,
		one has
		\begin{equation}\label{eq:MC-real-equivalence}
			c\,M_{\C}(F)\le M_2(F_1,F_2)\le C\,M_{\C}(F).
		\end{equation}
	\end{lemma}

	The optimal rate admits a useful formulation directly in terms of the complex kernel. This is especially convenient in applications, where contractions can often be estimated without computing moments explicitly. Fix $p,q\in\N_0$ with $l:=p+q\ge2$. If $p=q$, necessarily $p=q\ge1$, and we define
	\[
	\beta_{p,i}:=i!(p-i)!(p!)^2\binom pi^4,\qquad 0\le i\le p,
	\]
	and
	\[
	\mathcal T_1(f):=\left|\sum_{i=0}^p\beta_{p,i}\left\langle f\widetilde{\otimes_{i,p-i}}f,f^\star \right\rangle_{\mathfrak H^{\otimes 2p}}\right|,
	\qquad
	\mathcal T_2(f):=\left|\sum_{i=0}^p\beta_{p,i}\left\langle f\widetilde{\otimes_{i,p-i}}f,f\right\rangle_{\mathfrak H^{\otimes 2p}}\right|.
	\]
	If $p\ne q$, set $\mathcal T_1(f)=\mathcal T_2(f)=0$. Define the contraction rate
	\begin{equation}\label{eq:contraction-rate}
		\mathcal R_{p,q}(f):=\max\Bigg\{\mathcal T_1(f),\ \mathcal T_2(f),\ \sum_{\substack{0\le i\le p,\ 0\le j\le q\\0<i+j<l}}\|f\otimes_{i,j} f^\star \|_{\mathfrak H^{\otimes 2(l-i-j)}}^2\Bigg\}.
	\end{equation}

	\begin{lemma}\label{lem:MC-contraction}
		Fix $p,q\in\N_0$ with $p+q\ge2$ and let $F=I_{p,q}(f)$, where $f\in\mathfrak H^{\odot p}\otimes\mathfrak H^{\odot q}$. There exist constants $0<c\le C<\infty$, depending only on $p$ and $q$, such that
		\begin{equation}\label{eq:MC-contraction-equivalence}
			c\,\mathcal R_{p,q}(f)\le M_{\C}(F)\le C\,\mathcal R_{p,q}(f).
		\end{equation}
	\end{lemma}

	For a positive definite $2\times2$ matrix $\Sigma$ and a real $2\times2$ matrix $M=(m_{ij})$, we use
	\begin{equation}\label{eq:gaussian-density-HS}
		\phi_\Sigma(x):=\frac{1}{2\pi\sqrt{\det\Sigma}}\exp\left\{-\frac12x^\top\Sigma^{-1}x\right\},
		\qquad
		\|M\|_{\HS}:=\left(\sum_{i,j=1}^2m_{ij}^2\right)^{1/2}.
	\end{equation}
	We write $\phi=\phi_{I_2}$. We begin with the normalized case, where Theorem~\ref{thm:CCL-real-optimal} applies directly.
	
	\begin{proposition}\label{prop:normalized-optimal}
		Fix $p,q\in\N_0$ with $p+q\ge2$. For $n\geq1$, let $F_n=I_{p,q}(f_n)$, where $f_n\in\mathfrak H^{\odot p}\otimes\mathfrak H^{\odot q}$. Write $\widetilde F_n=(\mathrm{Re}\,F_n,\mathrm{Im}\,F_n)$. Assume 
		\begin{equation}\label{eq:normalized-assumptions}
			\Cov(\widetilde F_n)=I_2,
			\qquad
			\widetilde F_n\xrightarrow{\mathrm{law}}N_2(0,I_2), \qquad n\to\infty.
		\end{equation}
		Then, for every $m\in\N_0$ and every $r\in[1,\infty]$, there exists an integer $n_0$ such that, for every $n\ge n_0$, $\widetilde F_n$ has a density $p_{F_n}\in W^{m,r}(\R^2)$ and 
		\begin{equation}\label{eq:normalized-optimal-bound}
			\|p_{F_n}-\phi\|_{W^{m,r}(\R^2)}
			\asymp  d_\ast\bigl(\widetilde F_n,N_2(0,I_2)\bigr)\asymp 
			M_{\C}(F_n),
		\end{equation}
		where $d_\ast\in\{d_{\mathrm{TV}},d_{\mathrm{Kol}},d_{\rm W}\}$.
	\end{proposition}
	
	\begin{proof}
		By Proposition~\ref{prop:complex-real}, for every $n$ there exist $u_n,v_n\in\mathfrak K^{\odot(p+q)}$ such that
		\[
		\widetilde F_n
		=
		\bigl(I_{p+q}^W(u_n),I_{p+q}^W(v_n)\bigr).
		\]
		Therefore the structural assumption of Theorem~\ref{thm:CCL-real-optimal} is satisfied with $d=2$ and real chaos order $p+q$, while the covariance and convergence assumptions of that theorem are exactly \eqref{eq:normalized-assumptions}. Consequently, for all sufficiently large $n$, $\widetilde F_n$ has a density $p_{F_n}\in W^{m,r}(\R^2)$ and
		\[
		\|p_{F_n}-\phi\|_{W^{m,r}(\R^2)}\asymp d_\ast\bigl(\widetilde F_n,N_2(0,I_2)\bigr)\asymp M_2(\widetilde F_n).
		\]
		Lemma~\ref{lem:MC-real} gives $M_2(\widetilde F_n)\asymp M_{\C}(F_n)$ and proves \eqref{eq:normalized-optimal-bound}.
	\end{proof}
	
	To pass from identity covariance to a general positive definite covariance, we use two stability estimates. We state them in dimension two, which is the only case needed here.
	
	For a function $g$ on $\R^2$ and a matrix $A\in\R^{2\times2}$, define
	\[
	(T_Ag)(x):=|\det A|\,g(Ax),
	\qquad x\in\R^2.
	\]
	
	\begin{lemma}\label{lem:linear-change}
		Fix $K\ge1$. Let $A\in\R^{2\times2}$ be invertible and satisfy
		\begin{equation}\label{eq:A-uniform-condition}
			\|A\|_{\HS}+\|A^{-1}\|_{\HS}\le K.
		\end{equation} 
		Then the following statements hold.
		\begin{enumerate}
			\item Fix $k\ge2$ and let $\mathbf F=(F_1,F_2)\in (\mathcal H_k(W))^2$. There exist constants $0<c_1\le C_1<\infty$, depending only on $k$ and $K$, such that
			\begin{equation}\label{eq:M-linear-equivalence}
				c_1 M_2(\mathbf F)\le M_2(A\mathbf F)\le C_1 M_2(\mathbf F).
			\end{equation}
			\item For every $m\in\N_0$ and $r\in[1,\infty]$, there exist constants $0<c_2\le C_2<\infty$, depending only on $m,r$ and $K$, such that for $g\in W^{m,r}(\R^2) $,
			\begin{equation*}
				c_2 \|g\|_{W^{m,r}(\R^2)}
				\le
				\|T_Ag\|_{W^{m,r}(\R^2)}
				\le
				C_2\|g\|_{W^{m,r}(\R^2)}.
			\end{equation*}
			\item For arbitrary $\R^2$-valued random vectors $X,Y$ with finite first moments,
			\begin{equation}\label{eq:metric-linear}
				d_{\mathrm{TV}}(AX,AY)=d_{\mathrm{TV}}(X,Y),
				\qquad
				K^{-1}d_{\rm W}(X,Y)\le d_{\rm W}(AX,AY)\le Kd_{\rm W}(X,Y).
			\end{equation}
		\end{enumerate}
	\end{lemma}

	The proof of Lemma~\ref{lem:linear-change} is given in Subsection~\ref{subsec:auxiliary-proofs}. For the Kolmogorov distance and the fixed-target lower bound, we also need the following separation estimate. For $h\in L^1(\R^2)$, write
	\begin{align*}
		&\|h\|_{\mathrm{Kol}}
		:=
		\sup_{x\in\R^2}
		\left|
		\int_{(-\infty,x_1]\times(-\infty,x_2]}h(y)\,\mathrm dy
		\right|,\\
		&\|h\|_{\mathrm{BL}}
		:=
		\sup\left\{
		\left|\int_{\R^2}g(x)h(x)\,\mathrm dx\right|:
		g\in C_c^\infty(\R^2),\ \|g\|_\infty\le1,\ \operatorname{Lip}(g)\le1
		\right\}.
	\end{align*}
	If $\R^2$-valued random vectors $X$ and $Y$ have densities $p_X$ and $p_Y$, respectively, then
	\[
	\begin{aligned}
		d_{\mathrm{TV}}(X,Y)&=\frac12\|p_X-p_Y\|_{L^1(\R^2)},\qquad
		d_{\mathrm{Kol}}(X,Y)=\|p_X-p_Y\|_{\mathrm{Kol}},\\
		d_{\rm W}(X,Y)&\ge \|p_X-p_Y\|_{\mathrm{BL}}.
	\end{aligned}
	\]
	Set
	\[
	\mathcal V
	:=
	\Span\{\partial_\alpha\phi:\alpha\in[2]^3\cup[2]^4\}.
	\]
	For a positive definite $\Sigma\in\R^{2\times2}$ and a symmetric $H\in\R^{2\times2}$, define
	\begin{equation}\label{eq:cov-leading-correction}
		\mathcal L_\Sigma(H)(x)
		:=
		\frac12\phi_\Sigma(x)
		\left(
		x^\top \Sigma^{-1}H\Sigma^{-1}x-\operatorname{tr}(\Sigma^{-1}H)
		\right)
		=
		\frac12\sum_{i,j=1}^2H_{ij}\partial_{ij}\phi_\Sigma(x),
	\end{equation}
	which is the directional derivative of the Gaussian density $\phi_\Sigma$ with respect to the covariance matrix $\Sigma$ in the direction $H$.
	
	\begin{lemma}\label{lem:Gaussian-separation}
		Fix $K\ge1$, $m\in\N_0$ and $r\in[1,\infty]$, and let
		$\|\cdot\|_\ast$ be any one of
		\[
		\|\cdot\|_{W^{m,r}(\R^2)},\qquad
		\|\cdot\|_{L^1(\R^2)},\qquad
		\|\cdot\|_{\mathrm{Kol}},\qquad
		\|\cdot\|_{\mathrm{BL}}.
		\]
		There exists $c_\ast>0$, depending only on $K,m,r$ and the chosen norm,
		such that, for every invertible $A\in\R^{2\times2}$ satisfying
		\eqref{eq:A-uniform-condition}, every symmetric
		$H\in\R^{2\times2}$, and every $h\in\mathcal V$,
		\begin{equation*}
			\|\mathcal L_{\Sigma_A}(H)+T_Ah\|_\ast
			\ge
			c_\ast\bigl(\|H\|_{\HS}+\|h\|_\ast\bigr),
			\qquad
			\Sigma_A:=(A^\top A)^{-1}.
		\end{equation*}
	\end{lemma}

	The proof of Lemma~\ref{lem:Gaussian-separation} is deferred to Subsection~\ref{subsec:auxiliary-proofs}.

	\begin{theorem}\label{thm:general-optimal}
		Fix $p,q\in\N_0$ with $p+q\ge2$. For $n\geq1$, let $F_n=I_{p,q}(f_n)$, where $f_n\in\mathfrak H^{\odot p}\otimes\mathfrak H^{\odot q}$. Write $\widetilde F_n=(\mathrm{Re}\,F_n,\mathrm{Im}\,F_n)$. Assume 
		\begin{equation}\label{eq:general-assumptions}
			\Sigma_n:=\Cov(\widetilde F_n)\rightarrow \Sigma,
			\qquad
			\Sigma\text{ is positive definite},
			\qquad
			\widetilde F_n\xrightarrow{\mathrm{law}}N_2(0,\Sigma).
		\end{equation}
		Then, for every $m\in\N_0$ and every $r\in[1,\infty]$, there exists an integer $n_0$ such that, for every $n\ge n_0$, $\widetilde F_n$ has a density $p_{F_n}\in W^{m,r}(\R^2)$ and 
		\begin{equation}\label{eq:general-optimal-bound}
			\|p_{F_n}-\phi_{\Sigma_n}\|_{W^{m,r}(\R^2)}
			\asymp d_\ast\bigl(\widetilde F_n,N_2(0,\Sigma_n)\bigr) \asymp M_{\C}(F_n),
		\end{equation}
		where $d_\ast\in\{d_{\mathrm{TV}},d_{\mathrm{Kol}},d_{\rm W}\}$.
	\end{theorem}

	\begin{proof}
		Since $\Sigma_n\to\Sigma$ and $\Sigma$ is positive definite, $\Sigma_n$ is positive definite for all sufficiently large $n$. Set
		\[
		A_n:=\Sigma_n^{-1/2},
		\qquad
		G_n=(G_{n,1},G_{n,2}):=A_n\widetilde F_n.
		\]
		Then $\Cov(G_n)=I_2$. Since each $G_{n,i}$ is a real linear combination of the two coordinates of $\widetilde F_n$, Proposition~\ref{prop:complex-real} gives
		$G_n\in(\mathcal H_{p+q}(W))^2$. Moreover, $A_n\to \Sigma^{-1/2}$, and hence Slutsky's theorem and \eqref{eq:general-assumptions} yield
		\[
		G_n\xrightarrow{\mathrm{law}}N_2(0,I_2).
		\]
		Thus $G_n$ satisfies the assumptions of Theorem~\ref{thm:CCL-real-optimal} with dimension $d=2$ and fixed real chaos order $p+q$. Consequently, for all sufficiently large $n$, $G_n$ has a density $p_{G_n}\in W^{m,r}(\R^2)$ and
		\begin{align*}
			\|p_{G_n}-\phi\|_{W^{m,r}(\R^2)}\asymp d_\ast\bigl(G_n,N_2(0,I_2)\bigr)\asymp M_2(G_n),
			\qquad
			d_\ast\in\{d_{\mathrm{TV}},d_{\mathrm{Kol}},d_{\rm W}\}.
		\end{align*}
		Since $G_n=A_n\widetilde F_n$, the density transformation formula gives
		\[
		p_{F_n}=T_{A_n}p_{G_n},
		\qquad
		\phi_{\Sigma_n}=T_{A_n}\phi.
		\]
		The matrices $A_n$ satisfy \eqref{eq:A-uniform-condition} with a common constant for all sufficiently large $n$. Hence Lemma~\ref{lem:linear-change} yields
		\begin{align*}
			&M_2(G_n)\asymp M_2(\widetilde F_n),\qquad \|p_{F_n}-\phi_{\Sigma_n}\|_{W^{m,r}(\R^2)}
			\asymp
			\|p_{G_n}-\phi\|_{W^{m,r}(\R^2)},\\
			&d_\ast\bigl(\widetilde F_n,N_2(0,\Sigma_n)\bigr)
			\asymp
			d_\ast\bigl(G_n,N_2(0,I_2)\bigr),
			\qquad
			d_\ast\in\{d_{\mathrm{TV}},d_{\rm W}\}.
		\end{align*}
		Together with Lemma~\ref{lem:MC-real}, which gives
		$M_2(\widetilde F_n)\asymp M_{\C}(F_n)$, this proves \eqref{eq:general-optimal-bound} for the $W^{m,r}(\R^2)$ norm and $d_\ast\in\{d_{\mathrm{TV}},d_{\rm W}\}$. The Kolmogorov upper bound follows immediately from $d_{\mathrm{Kol}}\le d_{\mathrm{TV}}$.
		
		It remains to prove the Kolmogorov lower bound. By \cite[Proposition~5.6]{CCL2026}, for all sufficiently large $n$,
		\begin{align}\label{eq:expansion-normalized}
			p_{G_n}-\phi=\mathcal E_n+R_n,
		\end{align}
		where
		\[
		\mathcal E_n
		=
		-\frac16\sum_{\alpha\in[2]^3}\kappa_\alpha(G_n)\partial_\alpha\phi
		+\frac1{24}\sum_{\eta\in[2]^4}\kappa_\eta(G_n)\partial_\eta\phi
		\in\mathcal V,
		\]
		and
		\begin{equation}\label{eq:CCL-remainder-Kol}
			\|R_n\|_{W^{m,r}(\R^2)}+\|R_n\|_{L^1(\R^2)}
			\le
			C\bigl(M_2(G_n)^2+M_2(G_n)^{5/4}\bigr).
		\end{equation}
		Taking $H=0$ in Lemma~\ref{lem:Gaussian-separation} with $\|\cdot\|_\ast=\|\cdot\|_{\mathrm{Kol}}$, and using \cite[Lemma~5.7 and Eq.~(5.12)]{CCL2026}, we obtain, for all sufficiently large $n$,
		\[
		\|T_{A_n}\mathcal E_n\|_{\mathrm{Kol}}
		\ge
		c\|\mathcal E_n\|_{\mathrm{Kol}} \ge cM_2(G_n). 
		\]
		Since $p_{F_n}-\phi_{\Sigma_n}=T_{A_n}\mathcal E_n+T_{A_n}R_n$
		and $\|T_{A_n}R_n\|_{\mathrm{Kol}}\le\|T_{A_n}R_n\|_{L^1}=\|R_n\|_{L^1}$, using \eqref{eq:CCL-remainder-Kol}, we obtain
		\[
		d_{\mathrm{Kol}}\bigl(\widetilde F_n,N_2(0,\Sigma_n)\bigr) = \|p_{F_n}-\phi_{\Sigma_n}\|_{\mathrm{Kol}}
		\ge
		cM_2(G_n)-C\bigl(M_2(G_n)^2+M_2(G_n)^{5/4}\bigr).
		\]
		Finally, $G_n\xrightarrow{\mathrm{law}}N_2(0,I_2)$ in a fixed Wiener--It\^o chaos, so hypercontractivity gives convergence of the joint moments up to order four and hence $M_2(G_n)\to0$. Therefore the remainder is absorbed for all sufficiently large $n$, and
		\[
		d_{\mathrm{Kol}}\bigl(\widetilde F_n,N_2(0,\Sigma_n)\bigr)
		\ge \frac c2 M_2(G_n)
		\asymp M_{\C}(F_n).
		\]
		This completes \eqref{eq:general-optimal-bound} for $d_\ast=d_{\mathrm{Kol}}$.
	\end{proof}

	Combining Lemma~\ref{lem:MC-contraction} with Theorem~\ref{thm:general-optimal} gives the equivalent kernel formulation of the covariance-matched rate.
	\begin{corollary}\label{cor:contraction-optimal}
		Under the assumptions of Theorem~\ref{thm:general-optimal}, for every $m\in\N_0$, $r\in[1,\infty]$ and $d_\ast\in\{d_{\mathrm{TV}},d_{\mathrm{Kol}},d_{\rm W}\}$,
		\[
		\begin{aligned}
			\|p_{F_n}-\phi_{\Sigma_n}\|_{W^{m,r}(\R^2)}
			\asymp d_\ast \bigl(\widetilde F_n,N_2(0,\Sigma_n)\bigr)
			\asymp \mathcal R_{p,q}(f_n).
		\end{aligned}
		\]
	\end{corollary}

	The covariance-matched result isolates the intrinsic non-Gaussian error. To compare with the fixed limiting Gaussian law, one must additionally quantify the covariance perturbation and show that it cannot cancel the intrinsic third-order and fourth-order correction. We now complete the proof of Theorem~\ref{thm:intro-main}.
	
	\begin{proof}[Proof of Theorem~\ref{thm:intro-main}]
		We first compare the two Gaussian laws. Set
		\[
		\Sigma_{n,t}:=\Sigma+t(\Sigma_n-\Sigma),\qquad 0\le t\le1.
		\]
		Since $\Sigma_n\to\Sigma$ and $\Sigma$ is positive definite, for all sufficiently large $n$ the matrices $\Sigma_{n,t}$ are positive definite and $\Sigma_{n,t}^{-1}$ is uniformly bounded in $n$ and $t$. For fixed $x\in\R^2$, let
		\[
		g_{n,x}(t):=\phi_{\Sigma_{n,t}}(x).
		\]
		Differentiating $\Sigma_{n,t}\Sigma_{n,t}^{-1}=I_2$ and using Jacobi's formula yields
		\begin{equation}\label{eq:derivative-of-C}
			g_{n,x}'(t)
			=
			\frac12\phi_{\Sigma_{n,t}}(x)
			\left(
			x^\top \Sigma_{n,t}^{-1}(\Sigma_n-\Sigma)\Sigma_{n,t}^{-1}x
			-\operatorname{tr}(\Sigma_{n,t}^{-1}(\Sigma_n-\Sigma))
			\right).
		\end{equation}
		Thus, for every ordered multi-index $\alpha$ with $|\alpha|\le m$, applying $\partial_\alpha$ to $g_{n,x}'(t)$ with respect to $x$, we obtain 
		\[
		|\partial_\alpha g_{n,x}'(t)|
		\le
		C_m\|\Sigma_n-\Sigma\|_{\HS}(1+|x|^{m+2})e^{-c_0|x|^2},
		\qquad 0\le t\le1,
		\]
		for constants $c_0$ and $C_m$ independent of $n$ and $t$. Integrating in $t$ gives
		\begin{equation}\label{eq:Gaussian-first-order-upper}
			\|\phi_{\Sigma_n}-\phi_\Sigma\|_{W^{m,r}(\R^2)}
			\le
			C_{m,r}\|\Sigma_n-\Sigma\|_{\HS}.
		\end{equation}
		By Theorem~\ref{thm:general-optimal} and the triangle inequality, \eqref{eq:Gaussian-first-order-upper} proves the Sobolev upper bound in \eqref{eq:intro-fixed-target}. Taking $(m,r)=(0,1)$ gives the total variation and Kolmogorov upper bounds. For the $1$-Wasserstein distance, if $Z\sim N_2(0,I_2)$, then
		\[
		d_{\rm W}\bigl(N_2(0,\Sigma_n),N_2(0,\Sigma)\bigr)
		\le
		\E\bigl|(\Sigma_n^{1/2}-\Sigma^{1/2})Z\bigr|
		\le
		C\|\Sigma_n-\Sigma\|_{\HS}.
		\]
		This proves all upper bounds.
		
		We turn to the lower bounds. Still set $A_n:=\Sigma_n^{-1/2}$ and $G_n:=A_n\widetilde F_n$.
		Then $\Cov(G_n)=I_2$ and $G_n\to N_2(0,I_2)$ in law. With $\mathcal E_n$ and $R_n$ as in \eqref{eq:expansion-normalized},
		\[
		p_{G_n}-\phi=\mathcal E_n+R_n,
		\qquad
		\mathcal E_n\in\mathcal V,
		\]
		and
		\begin{equation}\label{eq:fixed-target-CCL-rem}
			\|R_n\|_{W^{m,r}(\R^2)}+\|R_n\|_{L^1(\R^2)}
			\le
			C_{m,r}\bigl(M_2(G_n)^2+M_2(G_n)^{5/4}\bigr).
		\end{equation}
		Meanwhile, we have
		\begin{equation}\label{eq:expand1}
			p_{F_n}-\phi_\Sigma = T_{A_n}(p_{G_n}-\phi)+(\phi_{\Sigma_n}-\phi_\Sigma)
			=
			T_{A_n}\mathcal E_n+T_{A_n}R_n
			+(\phi_{\Sigma_n}-\phi_\Sigma).
		\end{equation}
		From \eqref{eq:derivative-of-C},
		\[
		g_{n,x}'(1)=\mathcal L_{\Sigma_n}(\Sigma_n-\Sigma)(x).
		\]
		Moreover,
		\begin{align*}
			g_{n,x}''(t)
			=&
			\phi_{\Sigma_{n,t}}(x)
			\Bigg[
			\frac14
			\left(
			x^\top \Sigma_{n,t}^{-1}(\Sigma_n-\Sigma)\Sigma_{n,t}^{-1}x
			-\operatorname{tr}(\Sigma_{n,t}^{-1}(\Sigma_n-\Sigma))
			\right)^2\\
			&-x^\top \Sigma_{n,t}^{-1}(\Sigma_n-\Sigma)\Sigma_{n,t}^{-1}(\Sigma_n-\Sigma)\Sigma_{n,t}^{-1}x
			+\frac12\operatorname{tr}
			\bigl(\Sigma_{n,t}^{-1}(\Sigma_n-\Sigma)\Sigma_{n,t}^{-1}(\Sigma_n-\Sigma)\bigr)
			\Bigg].
		\end{align*}
		Taylor's formula around $t=1$ therefore gives
		\begin{equation}\label{eq:Gaussian-second-order}
			\phi_{\Sigma_n}-\phi_\Sigma
			=
			\mathcal L_{\Sigma_n}(\Sigma_n-\Sigma)+S_n,
			\qquad
			S_n(x):=-\int_0^1 t\,g_{n,x}''(t)\,\mathrm dt.
		\end{equation}
		The same Gaussian-polynomial estimate as in \eqref{eq:Gaussian-first-order-upper} gives, for every $m\in\N_0$ and $r\in[1,\infty]$,
		\begin{equation}\label{eq:Gaussian-second-order-rem}
			\|S_n\|_{W^{m,r}(\R^2)} + \|S_n \|_{L^1(\R^2)}
			\le
			C_{m,r}\|\Sigma_n-\Sigma\|_{\HS}^2.
		\end{equation}
		Combining \eqref{eq:expand1} and \eqref{eq:Gaussian-second-order},
		\begin{equation}\label{eq:fixed-target-expansion}
			p_{F_n}-\phi_\Sigma
			=
			\mathcal L_{\Sigma_n}(\Sigma_n-\Sigma)
			+T_{A_n}\mathcal E_n
			+T_{A_n}R_n+S_n.
		\end{equation}
		We now apply the same argument to the four norms $\|\cdot\|_{W^{m,r}}$, $\|\cdot\|_{L^1}$, $\|\cdot\|_{\mathrm{Kol}}$ and $\|\cdot\|_{\mathrm{BL}}$.
		Their restrictions to $\mathcal V$ are norms. Hence \cite[Lemma~5.7]{CCL2026} gives, for each of them $\|\cdot\|_{\ast}$,
		\begin{equation}\label{eq:fixed-target-leading-general}
			\|\mathcal E_n\|_\ast
			\ge
			c_\ast M_2(G_n)
		\end{equation}
		for all sufficiently large $n$. Since the matrices $A_n$ satisfy \eqref{eq:A-uniform-condition} with a common constant,
		Lemma~\ref{lem:Gaussian-separation}, together with \eqref{eq:fixed-target-leading-general}, yields
		\begin{equation}\label{eq:fixed-target-principal-lower}
			\|\mathcal L_{\Sigma_n}(\Sigma_n-\Sigma)+T_{A_n}\mathcal E_n\|_\ast
			\ge
			c_\ast\bigl(\|\Sigma_n-\Sigma\|_{\HS}+M_2(G_n)\bigr).
		\end{equation}
		For the Sobolev norm, Lemma~\ref{lem:linear-change}, \eqref{eq:fixed-target-CCL-rem} and \eqref{eq:Gaussian-second-order-rem} imply
		\[
		\|T_{A_n}R_n\|_{W^{m,r}(\R^2)}+\|S_n\|_{W^{m,r}(\R^2)}
		\le
		C_{m,r}
		\left(
		M_2(G_n)^2+M_2(G_n)^{5/4}+\|\Sigma_n-\Sigma\|_{\HS}^2
		\right).
		\]
		For the other three norms the same estimate follows from the above inequality with $(m,r)=(0,1)$ and $\|\cdot\|_{\mathrm{Kol}},\ \|\cdot\|_{\mathrm{BL}}\le\|\cdot\|_{L^1}$. Combining with \eqref{eq:fixed-target-expansion} and \eqref{eq:fixed-target-principal-lower} gives, for all sufficiently large $n$,
		\begin{align*}
			&\|p_{F_n}-\phi_\Sigma\|_\ast
			\ge \|\mathcal L_{\Sigma_n}(\Sigma_n-\Sigma)+T_{A_n}\mathcal E_n\|_\ast - \|T_{A_n}R_n\|_{\ast}-\|S_n\|_{\ast} \\
			\ge&\, c_{\ast}\bigl(\|\Sigma_n-\Sigma\|_{\HS}+M_2(G_n)\bigr) -C_{m,r}\left(
			M_2(G_n)^2+M_2(G_n)^{5/4}+\|\Sigma_n-\Sigma\|_{\HS}^2
			\right).
		\end{align*}
		Since $\|\Sigma_n-\Sigma\|_{\HS}\to0$ and $M_2(G_n)\to0$, for all sufficiently large $n$,
		\begin{align*}
			\|p_{F_n}-\phi_\Sigma\|_\ast
			&\ge
			\frac{c_{\ast}}{2}\bigl(\|\Sigma_n-\Sigma\|_{\HS}+M_2(G_n)\bigr).
		\end{align*}
		The choices of $\|\cdot\|_{\ast}$ as $\|\cdot\|_{W^{m,r}}$, $\|\cdot\|_{L^1}$, $\|\cdot\|_{\mathrm{Kol}}$ give the Sobolev, total variation and Kolmogorov lower bounds, respectively. The choice $\|\cdot\|_{\ast}=\|\cdot\|_{\mathrm{BL}}$, together with
		\[
		d_{\rm W}\bigl(\widetilde F_n,N_2(0,\Sigma)\bigr)
		\ge
		\|p_{F_n}-\phi_\Sigma\|_{\mathrm{BL}}
		\]
		gives the Wasserstein lower bound. Finally, Lemmas~\ref{lem:MC-real} and~\ref{lem:linear-change} yield $M_2(G_n)\asymp M_{\C}(F_n)$,
		and this completes the proof of Theorem~\ref{thm:intro-main}.
	\end{proof}
	
	\begin{remark}\label{rem:complex-vector-extension}
		Corollary~\ref{cor:contraction-optimal} concerns a single complex multiple Wiener--It\^o integral, identified with a two-dimensional real vector. Extending the optimal density and non-smooth bounds to a complex vector
		\[
		\bigl(I_{p_1,q_1}(f_1),\ldots,I_{p_d,q_d}(f_d)\bigr)
		\]
		requires additional structure. In particular, the contraction mechanisms appearing in the complex Berry--Ess\'een analysis of \cite{ChenChenLiu2025Berry} indicate that the relevant mixed contractions are governed not only by their total orders but also by the partial-order relations among the coordinate bidegrees $(p_j,q_j)$. This feature is specific to the complex setting and has no direct analogue for real vectors whose coordinates lie in one fixed chaos. A systematic treatment of these partial-order effects, together with the corresponding optimal Sobolev and non-smooth bounds for complex Wiener--It\^o vectors, will be developed in future work.
	\end{remark}

	\subsection{Auxiliary estimates for the optimal bounds}\label{subsec:auxiliary-proofs}
	
	This subsection contains the proofs of the technical stability and separation estimates used in Subsection~\ref{sec:optimal-gaussian-approximation}.

	\begin{proof}[Proof of Lemma~\ref{lem:MC-real}]
		Chen \cite[Lemma~4.3]{Chen2024Optimal} proves that
		\[
		\kappa_4(F_1)+\kappa_4(F_2)
		\le
		\Delta_{4,\C}(F)
		\le
		C\bigl(\kappa_4(F_1)+\kappa_4(F_2)\bigr).
		\]
		Hence the fourth-order parts of $M_2(F_1,F_2)$ and $M_{\C}(F)$ are equivalent. For the third-order part, \cite[Lemma~4.4]{Chen2024Optimal} gives
		\begin{align*}
			\frac14\bigl(|\E[F^3]|+|\E[F^2\overline F]|\bigr)
			&\le
			|\E[F_1^3]|+|\E[F_2^3]|+|\E[F_1^2F_2]|+|\E[F_1F_2^2]|\\
			&\le
			\sqrt2\bigl(|\E[F^3]|+|\E[F^2\overline F]|\bigr).
		\end{align*}
		In a fixed finite-dimensional space, the maximum norm and the sum norm are equivalent. Combining the third-order and fourth-order comparisons proves \eqref{eq:MC-real-equivalence}.
	\end{proof}
	
	\begin{proof}[Proof of Lemma~\ref{lem:MC-contraction}]
		Suppose first that $p,q\in\N$. By \cite[Lemma~3.2]{ChenChenLiu2025Berry}, $\mathcal T_1(f)=|\E[F^3]|$ and $\mathcal T_2(f)=|\E[F^2\overline F]|$. Moreover, \cite[Proposition~3.2]{ChenChenLiu2025Berry} gives
		\begin{align}\label{eq:kernel_exp_fourth_1}
			\Delta_{4,\C}(F)\asymp
			\sum_{\substack{0\le i\le p,\ 0\le j\le q\\0<i+j<p+q}}
			\|f\otimes_{i,j}f^\star\|_{\mathfrak H^{\otimes 2(p+q-i-j)}}^2,
		\end{align}
		where the comparison constants depend only on $p$ and $q$. It remains to consider the boundary cases. If $q=0$ and $p\ge2$, then the complex product formula \eqref{eq:complex-product-formula} and isometry property \eqref{eq:complex-isometry} give
		\[
		\E[F^3]=0,\qquad \E[F^2\overline F]=0.
		\]
		The fourth-moment expansion used in the proof of \cite[Proposition~3.2]{ChenChenLiu2025Berry} reduces to
		\begin{align*}
			\Delta_{4,\C}(F)
			=(p!)^2\sum_{i=1}^{p-1}\binom pi^2
			\|f\otimes_{i,0}f^\star\|_{\mathfrak H^{\otimes 2(p-i)}}^2.
		\end{align*}
		Similarly, if $p=0$ and $q\ge2$, then
		\[
		\E[F^3]=0,\qquad \E[F^2\overline F]=0,\qquad \Delta_{4,\C}(F)
		=(q!)^2\sum_{j=1}^{q-1}\binom qj^2
		\|f\otimes_{0,j}f^\star\|_{\mathfrak H^{\otimes 2(q-j)}}^2.
		\]
		Thus in both boundary cases $M_{\C}(F)=\Delta_{4,\C}(F)$ and the latter is comparable, with constants depending only on the nonzero chaos order, to the contraction sum in \eqref{eq:contraction-rate}. Combining this with the case $p,q\in\N$ proves \eqref{eq:MC-contraction-equivalence} for all $p,q\in\N_0$ with $p+q\ge2$.
	\end{proof}

	\begin{proof}[Proof of Lemma~\ref{lem:linear-change}]
		(1) Every third-order moment of $A\mathbf F$ is a linear combination of the four third-order moments appearing in \eqref{eq:Md-real-prelim} with dimension $d=2$, with coefficients bounded in terms of $K$. Hence the third-order part of $M_2(A\mathbf F)$ is bounded by $CM_2(\mathbf F)$. For the fourth-order part, write $F_i=I_k^W(f_i)$ with $f_i\in\mathfrak K^{\odot k}$, $i=1,2$. For $a=(a_1,a_2)\in\R^2$,
		\[
		a_1F_1+a_2F_2
		=
		I_k^W(a_1f_1+a_2f_2).
		\]
		The fourth cumulant $\kappa_4(a_1F_1+a_2F_2)$ is a finite positive linear combination, with coefficients depending only on $k$, of squared contraction norms of its kernel; see \cite[Lemma~5.2.4]{nourdin2012normal}. Expanding the contractions of $a_1f_1+a_2f_2$ and applying Cauchy--Schwarz to each mixed contraction show that
		\[
		\kappa_4(a_1F_1+a_2F_2)
		\le
		C_k|a|^4
		\max\{\kappa_4(F_1),\kappa_4(F_2)\}.
		\]
		Applying this estimate to the two rows of $A$ gives that the fourth-order part of $M_2(A\mathbf F)$ is bounded by $CM_2(\mathbf F)$. Thus
		\[
		M_2(A\mathbf F)\le CM_2(\mathbf F).
		\]
		Since $\mathbf F=A^{-1}(A\mathbf F)$ and $A^{-1}$ also satisfies \eqref{eq:A-uniform-condition}, applying the same argument to $A^{-1}$ yields the reverse inequality and proves \eqref{eq:M-linear-equivalence}.
		
		(2) For the Sobolev norms, repeated application of the chain rule shows that, for every ordered partial derivative $\partial_\alpha$ with $|\alpha|\le m$, $\partial_\alpha(T_Ag)(x)$ is $|\det A|$ times a finite linear combination of derivatives $(\partial_\beta g)(Ax)$ with $|\beta|=|\alpha|$, whose coefficients are bounded in terms of $m$ and $K$. Moreover, \eqref{eq:A-uniform-condition} implies 
		\[
		K^{-2}\le|\det A|\le K^2.
		\]
		For $1\le r<\infty$, the change of variables $y=Ax$ therefore gives
		\[
		\|T_Ag\|_{W^{m,r}(\R^2)}
		\le
		C\|g\|_{W^{m,r}(\R^2)}.
		\]
		The same bound is immediate for $r=\infty$. Applying the estimate to $A^{-1}$ and using $T_{A^{-1}}T_Ag=g$ gives the reverse inequality.
		
		(3) The total variation identity in \eqref{eq:metric-linear} follows from the fact that an invertible linear map is a measurable bijection of $\R^2$. For the Wasserstein distance, if $h$ is $1$-Lipschitz, then $h\circ A$ is $\|A\|_{\HS}$-Lipschitz.
		Hence
		\[
		d_{\rm W}(AX,AY)\le\|A\|_{\HS}d_{\rm W}(X,Y)\le Kd_{\rm W}(X,Y).
		\]
		Applying the same estimate to $A^{-1}$ gives the reverse inequality in \eqref{eq:metric-linear}.
	\end{proof}
	
	\begin{proof}[Proof of Lemma~\ref{lem:Gaussian-separation}]
		First note that $\|\cdot\|_{W^{m,r}(\R^2)}$ and $\|\cdot\|_{L^1(\R^2)}$ are norms on $\mathcal V$. The same is true for $\|\cdot\|_{\mathrm{Kol}}$. Indeed, if $\|h\|_{\mathrm{Kol}}=0$ for $h\in L^1(\R^2)$, then
		\[
		\int_{(-\infty,x_1]\times(-\infty,x_2]} h(y)\,\mathrm dy=0,
		\qquad x\in\R^2,
		\]
		and applying $\partial_1\partial_2$ in the sense of distributions gives $h=0$. The quantity $\|\cdot\|_{\mathrm{BL}}$ is also a norm on $\mathcal V$. Indeed, if $\|h\|_{\mathrm{BL}}=0$ for $h\in L^1(\R^2)$, rescaling arbitrary functions in $C_c^\infty(\R^2)$ shows that $\int gh=0$ for every $g\in C_c^\infty(\R^2)$, hence again $h=0$ as a distribution.
		Let
		\[
		\mathcal A_K
		:=
		\left\{
		A\in\R^{2\times2}:
		A\text{ is invertible and }
		\|A\|_{\HS}+\|A^{-1}\|_{\HS}\le K
		\right\}.
		\]
		If $\mathcal A_K$ is empty, the assertion is vacuous. Otherwise, $\mathcal A_K$ is compact. Let
		\[
		\mathbb S_2:=\{H\in\R^{2\times2}:H^\top=H\},
		\qquad
		\mathcal S_\ast
		:=
		\{(H,h)\in\mathbb S_2\times\mathcal V:
		\|H\|_{\HS}+\|h\|_\ast=1\}.
		\]
		The set $\mathcal S_\ast$ is also compact. Define
		\[
		\Psi:\mathcal A_K\times\mathcal S_\ast \rightarrow [0,\infty), \qquad \Psi(A,H,h)
		:=
		\|\mathcal L_{\Sigma_A}(H)+T_Ah\|_\ast.
		\]
		The map $\Psi$ is continuous. To see this, let $(A_n,H_n,h_n)\rightarrow(A,H,h)$ in $\mathcal A_K\times\mathcal S_\ast$.
		For every ordered multi-index $\beta$ with $|\beta|\le m$,
		\[
		\partial_\beta
		\bigl(
		\mathcal L_{\Sigma_{A_n}}(H_n)+T_{A_n}h_n
		\bigr)(x)
		=
		P_{\beta,n}(x)e^{-|A_nx|^2/2},
		\]
		where the degrees of the polynomials $P_{\beta,n}$ are at most $m+4$ and their coefficients converge to those of the
		corresponding polynomial $P_\beta$ associated with $(A,H,h)$.
		Since $A_n\to A$ and all $A_n,A$ belong to $\mathcal A_K$, there
		exist $c>0$, $M<\infty$, and a sequence $\varepsilon_n\to0$ such
		that
		\[
		\left|
		\partial_\beta
		\bigl(
		\mathcal L_{\Sigma_{A_n}}(H_n)+T_{A_n}h_n
		-
		\mathcal L_{\Sigma_A}(H)-T_Ah
		\bigr)(x)
		\right|
		\le
		\varepsilon_n(1+|x|)^M e^{-c|x|^2}.
		\]
		Indeed, this follows from the convergence of the polynomial
		coefficients and the mean value theorem applied to
		$A\mapsto e^{-|Ax|^2/2}$, together with
		$|A_nx|\ge K^{-1}|x|$.
		The right-hand side belongs to every $L^r(\R^2)$,
		$1\le r\le\infty$. Therefore
		\[
		\mathcal L_{\Sigma_{A_n}}(H_n)+T_{A_n}h_n
		\rightarrow
		\mathcal L_{\Sigma_A}(H)+T_Ah
		\]
		in $W^{m,r}(\R^2)$ and in $L^1(\R^2)$. Since $\|\cdot \|_{\mathrm{Kol}}\le\|\cdot \|_{L^1}$ and $\|\cdot\|_{\mathrm{BL}}\le\|\cdot\|_{L^1}$,
		the same convergence holds for the norms $\|\cdot \|_{\mathrm{Kol}}$ and $\|\cdot\|_{\mathrm{BL}}$. Thus $\Psi$ is continuous on
		$\mathcal A_K\times\mathcal S_\ast$.
		
		It remains to show that $\Psi$ is strictly positive. Suppose that $\Psi(A,H,h)=
		\|\mathcal L_{\Sigma_A}(H)+T_Ah\|_\ast=0$ for some $(A,H,h)\in \mathcal A_K\times\mathcal S_\ast$. Then $\mathcal L_{\Sigma_A}(H)+T_Ah=0$. Since
		$\phi_{\Sigma_A}=T_A\phi$, the chain rule shows that $T_Ah$ is a linear
		combination of third-order and fourth-order derivatives of $\phi_{\Sigma_A}$,
		whereas $\mathcal L_{\Sigma_A}(H)$ is a linear combination of second-order derivatives of $\phi_{\Sigma_A}$ by \eqref{eq:cov-leading-correction}. After division by $\phi_{\Sigma_A}$, $\mathcal L_{\Sigma_A}(H)$ belongs to the second Gaussian Hermite space, whereas $T_Ah$ belongs to the direct sum of the third and
		fourth Gaussian Hermite spaces. Since Gaussian Hermite spaces of different orders form a direct sum,
		\[
		\mathcal L_{\Sigma_A}(H)+T_Ah=0
		\]
		implies $\mathcal L_{\Sigma_A}(H)=0$ and $T_Ah=0$. Since $T_A$ is invertible,
		$T_Ah=0$ gives $h=0$. Moreover, the quadratic part of
		$\mathcal L_{\Sigma_A}(H)/\phi_{\Sigma_A}$ is
		\[
		\frac12 x^\top \Sigma_A^{-1}H\Sigma_A^{-1}x,
		\]
		so $\mathcal L_{\Sigma_A}(H)=0$ gives $H=0$.
		This contradicts $(H,h)\in\mathcal S_\ast$. Therefore $\Psi>0$ on the
		compact set $\mathcal A_K\times\mathcal S_\ast$, and
		\[
		c_\ast
		:=
		\min_{\mathcal A_K\times\mathcal S_\ast}\Psi
		>0.
		\]
		Therefore, for every $A\in\mathcal A_K$, $H\in\mathbb S_2$ and
		$h\in\mathcal V$, if $H$ or $h$ is non-zero, applying the preceding estimate to
		\[
		\frac{(H,h)}{\|H\|_{\HS}+\|h\|_\ast}\in\mathcal S_\ast
		\]
		gives
		\[
		\|\mathcal L_{\Sigma_A}(H)+T_Ah\|_\ast
		\ge
		c_\ast\bigl(\|H\|_{\HS}+\|h\|_\ast\bigr).
		\]
		The case $H=0$ and $h=0$ is immediate, and the proof is complete.
	\end{proof}
	
	\section{Application to a complex fractional Ornstein--Uhlenbeck process}
	\label{sec:app}

	As an application, we study the asymptotic behavior of the normalized numerator in the least-squares estimator of a complex fractional Ornstein--Uhlenbeck process arising from the classical Chandler wobble model. This statistic is a complex $(1,1)$-th Wiener--It\^o integral, so the general results of Section~\ref{sec:main-result} reduce the approximation problem to identifying the intrinsic moment scale and the covariance mismatch. We state the resulting optimal rates first and then establish the quantitative estimates needed for their proof.

	\subsection{Optimal Gaussian approximation}

	Throughout this section,
	\[
	H\in\left(\frac12,\frac34\right),
	\qquad
	\alpha_H:=H(2H-1),\qquad \beta:=2H-2,
	\qquad
	\R_+:=[0,\infty).
	\]
	Let $\mathcal E_{\C}$ be the space of complex-valued step
	functions on $\R_+$. We define $\mathfrak H_H$ as the
	completion of $\mathcal E_{\C}$ with respect to the
	inner product
	\begin{equation}\label{eq:fBm-inner}
		\langle f,g\rangle_{\mathfrak H_H}
		:=
		\alpha_H
		\int_{\R_+^2}
		f(s)\overline{g(t)}
		|t-s|^{\beta}\,\mathrm ds\,\mathrm dt.
	\end{equation}
	Then $\mathfrak H_H$ is a complex separable Hilbert space. Let $Z_H=\{Z_H(h):h\in\mathfrak H_H\}$
	be a complex isonormal Gaussian process over $\mathfrak H_H$.
	We define $\zeta_t
	:=
	Z_H(\mathbf 1_{[0,t]})$ for $t\ge0$.
	Then $\zeta=\{\zeta_t:t\ge0\}$ is a complex fractional
	Brownian motion. Indeed,
	\begin{equation*}
		\E[\zeta_t\zeta_s]=0,\qquad \E[\zeta_t\overline{\zeta_s}]
		=
		R_H(t,s)
		:=
		\frac12
		\left(
		t^{2H}+s^{2H}-|t-s|^{2H}
		\right).
	\end{equation*}
	
	Consider the complex fractional Ornstein--Uhlenbeck process
	\begin{equation*}
		dZ_t
		=
		-\gamma Z_t\,\mathrm dt+d\zeta_t,
		\qquad
		Z_0=0,
	\end{equation*}
	where $\gamma\in\C$ satisfies $\lambda:=\mathrm{Re}\,\gamma>0$. 
	Following \cite{HU2010}, Chen, Hu and Wang \cite{ChenHuWang2017} considered the least-squares estimator of $\gamma$, formally obtained by minimizing $\int_0^T \bigl|\dot Z_t+\gamma Z_t\bigr|^2\,\mathrm dt$.
	It is given by
	\begin{equation*}
		\widehat\gamma_T
		=
		\gamma-
		\frac{
			\int_0^T\overline{Z_t}\,\mathrm d\zeta_t
		}{
			\int_0^T|Z_t|^2\,\mathrm dt
		}.
	\end{equation*}
	We concentrate on the normalized numerator of $\widehat\gamma_T-\gamma$, which can be represented as a complex $(1,1)$-th Wiener--It\^o integral
	\begin{equation*}
		F_T
		:=
		\frac1{\sqrt T}
		\int_0^T\overline{Z_t}\,\mathrm d\zeta_t = I_{1,1}(T^{-1/2}\psi_T),
		\qquad
		\psi_T(t,s)
		:=
		e^{-\overline\gamma(t-s)}
		\mathbf 1_{\{0\le s\le t\le T\}}.
	\end{equation*}
	By Proposition~\ref{prop:complex-real}, $\widetilde F_T:=
	\bigl(\mathrm{Re}\,F_T,\mathrm{Im}\,F_T\bigr)$
	has both coordinates in the second real Wiener--It\^o chaos associated
	with the direct-sum space $\mathfrak H_{H,\R}\oplus\mathfrak H_{H,\R}$, where $\mathfrak H_{H,\R}
	:= \left\lbrace f\in \mathfrak H_H: \overline{f}=f \right\rbrace $
	is the associated real Hilbert space of $\mathfrak H_H$.
	
	For later use, set $v_T:=\E|F_T|^2$ and $\eta_T:=\E[F_T^2]$.
	The covariance matrix of $\widetilde F_T$ can then be written as
	\begin{equation}\label{eq:CT}
		\Sigma_T
		:=
		\Cov(\widetilde F_T)
		=
		\frac12
		\begin{pmatrix}
			v_T+\mathrm{Re}\,\eta_T
			&
			\mathrm{Im}\,\eta_T
			\\
			\mathrm{Im}\,\eta_T
			&
			v_T-\mathrm{Re}\,\eta_T
		\end{pmatrix}.
	\end{equation}
	The covariance calculations and the central limit theorem in
	\cite{ChenHuWang2017} imply that there exist
	$v_\infty\in(0,\infty)$ and $\eta_\infty\in\C$ such that $v_T\rightarrow v_\infty$ and $\eta_T\rightarrow\eta_\infty$ as $T\rightarrow\infty$.
	Then
	\begin{equation}\label{eq:cfou-CLT}
		\widetilde F_T
		\xrightarrow[T\to\infty]{\mathrm{law}}
		N_2(0,\Sigma_\infty),\qquad 
		\Sigma_\infty
		:=
		\frac12
		\begin{pmatrix}
			v_\infty+\mathrm{Re}\,\eta_\infty
			&
			\mathrm{Im}\,\eta_\infty
			\\
			\mathrm{Im}\,\eta_\infty
			&
			v_\infty-\mathrm{Re}\,\eta_\infty
		\end{pmatrix}.
	\end{equation}
	Explicit expressions for $v_\infty$ and $\eta_\infty$ are given in \cite{ChenHuWang2017}; those formulas also show that $\Sigma_\infty$ is positive definite.

	To state the quantitative estimates, define
	\begin{equation*}
		\rho_H(T)
		:=
		\begin{cases}
			T^{-1/2},
			& H\in(1/2,5/8],
			\\[1mm]
			T^{4H-3},
			& H\in(5/8,3/4).
		\end{cases}
	\end{equation*}
	We also use the scale
	\begin{equation*}
		\tau_H(T)
		:=
		\begin{cases}
			T^{-1/2},
			& H\in(1/2,2/3),
			\\[1mm]
			T^{-1/2}\log T,
			& H=2/3,
			\\[1mm]
			T^{6H-9/2},
			& H\in(2/3,3/4).
		\end{cases}
	\end{equation*}
	Notice that $\rho_H(T),\tau_H(T)\to0$ as $T\to\infty$. For $T\ge2$, a direct comparison of the exponents gives
	\begin{equation}\label{eq:rho-tau-comparison}
		\tau_H(T)\le C_H\rho_H(T),
		\qquad
		\rho_H(T)^2\le C_H\tau_H(T).
	\end{equation}

	Let $\phi_T:=\phi_{\Sigma_T}$ and $\phi_\infty:=\phi_{\Sigma_\infty}$ denote the centered Gaussian densities on $\R^2$ with covariance matrices $\Sigma_T$ and $\Sigma_\infty$,
	respectively. The next result gives the explicit optimal rates for both the covariance-matched and fixed-target approximations.
	
	\begin{theorem}\label{thm:cfou-density}
		For every $m\in\N_0$ and every $r\in[1,\infty]$, there exists $T_0\ge2$ such that, for every $T\ge T_0$,
		$\widetilde F_T$ has a density $p_T\in W^{m,r}(\R^2)$ and, for each
		$d_\ast\in\{d_{\mathrm{TV}},d_{\mathrm{Kol}},d_{\rm W}\}$,
		\begin{align}
			\|p_T-\phi_T\|_{W^{m,r}(\R^2)}
			&\asymp
			d_\ast\bigl(\widetilde F_T,N_2(0,\Sigma_T)\bigr)
			\asymp
			M_{\C}(F_T)
			\asymp
			\tau_H(T),
			\label{eq:cfou-optimal-CT}\\
			\|p_T-\phi_\infty\|_{W^{m,r}(\R^2)}
			&\asymp
			d_\ast\bigl(\widetilde F_T,N_2(0,\Sigma_\infty)\bigr)
			\asymp
			M_{\C}(F_T)+\|\Sigma_T-\Sigma_\infty\|_{\HS}
			\asymp
			\rho_H(T).
			\label{eq:cfou-explicit-Cinfty}
		\end{align}
	\end{theorem}
	
	Theorem~\ref{thm:cfou-density} shows that the covariance-matched and fixed-target approximations need not have the same rate. The two rates $\tau_H(T)$ and $\rho_H(T)$ have different origins. The covariance-matched rate $\tau_H(T)$ is governed by the intrinsic third-order and fourth-order moment terms, whereas the fixed-target rate $\rho_H(T)$ also records the covariance error. In particular, the transition at $H=2/3$ comes from the third-order moments, while the transition at $H=5/8$ comes from the covariance mismatch. The proof of the theorem is given after the estimates in Subsection~\ref{subsec:cfou-estimates}. 
	
	\begin{remark}[The Brownian case]
		The endpoint $H=1/2$ is not covered by the Hilbert-space
		representation \eqref{eq:fBm-inner} and should be treated separately.
		In the complex Brownian case, \cite[Section~4]{Chen2024Optimal} gives
		\[
		\E[F_T^2]=0,
		\qquad
		\E[|F_T|^2]
		=
		\frac{1}{2\lambda}
		+\frac{e^{-2\lambda T}-1}{4\lambda^2T},\qquad \widetilde F_T
		\xrightarrow[T\to\infty]{\mathrm{law}}
		N_2(0,\Sigma_\infty),
		\qquad
		\Sigma_\infty=\frac{1}{4\lambda}I_2.
		\]
		Hence
		\[
		\Sigma_T
		=
		\frac12\E[|F_T|^2]I_2
		=
		\left(
		\frac{1}{4\lambda}
		+\frac{e^{-2\lambda T}-1}{8\lambda^2T}
		\right)I_2,\qquad \|\Sigma_T-\Sigma_\infty\|_{\HS}\asymp T^{-1}.
		\]
		Moreover, the explicit computations in
		\cite[Lemmas~4.6--4.7]{Chen2024Optimal} show that
		\[
		\E[F_T^3]=0,
		\qquad
		|\E[F_T^2\overline{F_T}]|
		\asymp T^{-1/2},
		\qquad
		\Delta_{4,\C}(F_T)
		\asymp T^{-1}.
		\]
		Therefore, for every $m\in\N_0$, $r\in[1,\infty]$, and
		$d_\ast\in\{d_{\mathrm{TV}},d_{\mathrm{Kol}},d_{\rm W}\}$,
		Theorem~\ref{thm:general-optimal} yields
		\[
		\|p_T-\phi_{\Sigma_T}\|_{W^{m,r}(\R^2)}
		\asymp
		d_\ast\bigl(\widetilde F_T,N_2(0,\Sigma_T)\bigr)
		\asymp M_{\C}(F_T)\asymp 
		T^{-1/2}.
		\]
		On the other hand, Theorem~\ref{thm:intro-main} gives
		\begin{align*}
			\|p_T-\phi_{\Sigma_\infty}\|_{W^{m,r}(\R^2)}
			&\asymp
			d_\ast\bigl(\widetilde F_T,N_2(0,\Sigma_\infty)\bigr)\asymp
			M_{\C}(F_T)+\|\Sigma_T-\Sigma_\infty\|_{\HS}
			\asymp
			T^{-1/2}.
		\end{align*}
		Thus, in the Brownian case, both the covariance-matched approximation
		and the approximation to the limiting Gaussian distribution have the
		same optimal rate $T^{-1/2}$. This also complements the optimal rates for smooth and $1$-Wasserstein distances obtained in \cite[Theorem~4.5]{Chen2024Optimal} and \cite[Theorem~4.1]{ChenChenLiu2025Berry}, respectively.
	\end{remark}

	\subsection{Quantitative estimates}\label{subsec:cfou-estimates}
	
	We now establish the covariance and moment estimates that determine the two rates in Theorem~\ref{thm:cfou-density}.
	
	\begin{lemma}\label{lem:cfou-quant}
		There exist constants $0<c_{\gamma,H}\le C_{\gamma,H}<\infty$ and $T_0\ge1$ such that, for every $T\ge T_0$,
		\begin{align}
			&c_{\gamma,H}T^{4H-3} \le \|\Sigma_T-\Sigma_\infty\|_{\HS}
			\le
			C_{\gamma,H}T^{4H-3},\label{eq:cfou-cov-rate}\\
			&\sqrt{\Delta_{4,\C}(F_T)}
			\le
			C_{\gamma,H}\rho_H(T).\label{eq:cfou-defect-rate}
		\end{align}
	\end{lemma}
	
	\begin{proof}
		For \eqref{eq:cfou-defect-rate}, by \eqref{eq:kernel_exp_fourth_1}, it suffices to bound
		\[
		T^{-1}\left(
		\|\psi_T\otimes_{1,0}\psi_T^\star\|_{\mathfrak H_{H}^{\otimes2}}
		+
		\|\psi_T\otimes_{0,1}\psi_T^\star\|_{\mathfrak H_{H}^{\otimes2}}
		\right).
		\]
		Using the contraction estimates in
		\cite[Lemma~3.6]{ChenHuWang2017} and
		\cite[Lemma~3.1]{ChenKuangLi2020}; see also
		\cite[Remark~4.1]{ChenChenLiu2025Berry}, we obtain
		\eqref{eq:cfou-defect-rate}.
		
		We now prove \eqref{eq:cfou-cov-rate}. Recall that
		$\beta=2H-2\in(-1,-1/2)$. By \eqref{eq:complex-isometry} and
		\eqref{eq:fBm-inner},
		\begin{equation}\label{eq:vT-fourfold}
			v_T
			=\frac{\alpha_H^2}{T}
			\int_{\{0\le s_1\le t_1\le T,\ 0\le s_2\le t_2\le T\}}
			e^{-\overline\gamma(t_1-s_1)}e^{-\gamma(t_2-s_2)}
			|t_1-t_2|^\beta |s_1-s_2|^\beta
			\,\mathrm dt_1\mathrm ds_1\mathrm dt_2\mathrm ds_2.
		\end{equation}
		Put $u=t_1-s_1$, $w=t_2-s_2$, and $r=s_1-s_2$. For fixed $u,w\ge0$ and $r\in\R$, the admissible values of $s_2$ form the
		interval
		\[
		\max\{0,-r\}
		\le s_2\le
		\min\{T-w,T-u-r\}.
		\]
		Accordingly, for $0\le u,w\le T$, define
		\begin{equation}\label{eq:qT-definition}
			q_T(r;u,w)
			:=
			\frac1T
			\Bigl[
			\min\{T-w,T-u-r\}-\max\{0,-r\}
			\Bigr]_+,
		\end{equation}
		and set $q_T(r;u,w)=0$ if $u>T$ or $w>T$, where, throughout this proof, $[x]_+:=\max\{x,0\}$. Thus $0\le q_T\le1$, and
		integrating first with respect to the translation variable $s_2$ gives
		\begin{equation}\label{eq:vT-overlap}
			v_T
			=
			\alpha_H^2
			\int_{\R_+^2}e^{-\overline\gamma u-\gamma w}
			\int_{\R}|r|^\beta|r+u-w|^\beta q_T(r;u,w)
			\,\mathrm dr\,\mathrm du\,\mathrm dw.
		\end{equation}
		The limiting integral is absolutely convergent. Indeed,
		\[
		\int_{\R}|r|^\beta|r+d|^\beta\,\mathrm dr
		=C_\beta |d|^{2\beta+1},\qquad d\ne0,
		\]
		where $C_\beta<\infty$ because $\beta>-1$ and $2\beta<-1$, while
		$2\beta+1\in(-1,0)$. Since $0\le q_T\le1$, this absolute convergence provides an integrable dominator in $(r,u,w)$. Together with $q_T(r;u,w)\to1$ pointwise, the dominated convergence theorem yields
		\[
		v_\infty
		=
		\alpha_H^2
		\int_{\R_+^2}e^{-\overline\gamma u-\gamma w}
		\int_{\R}|r|^\beta|r+u-w|^\beta
		\,\mathrm dr\,\mathrm du\,\mathrm dw.
		\]
		Subtracting \eqref{eq:vT-overlap}, setting $r=Tx$, and using
		$2\beta+1=4H-3$, we obtain
		\begin{equation}\label{eq:vT-scaled-error}
			\begin{aligned}
				T^{3-4H}(v_\infty-v_T)
				&=
				\alpha_H^2
				\int_{\R_+^2}e^{-\overline\gamma u-\gamma w}
				\int_{\R}|x|^\beta
				\left|x+\frac{u-w}{T}\right|^\beta
				\bigl(1-q_T(Tx;u,w)\bigr)
				\,\mathrm dx\,\mathrm du\,\mathrm dw.
			\end{aligned}
		\end{equation}
		For every fixed $u,w\ge0$ and $x\ne0$, the overlap length in
		\eqref{eq:qT-definition} gives
		\begin{equation}\label{eq:qT-limit}
			q_T(Tx;u,w)\rightarrow(1-|x|)_+,
		\end{equation}
		and
		\begin{equation}\label{eq:qT-bound}
			1-q_T(Tx;u,w)
			\le
			\min\left\{1,|x|+\frac{u+w}{T}\right\}.
		\end{equation}
		We briefly record the uniform-integrability estimate needed below. Fix
		$R>0$ and let $T\ge2R$. On $0\le u,w\le R$ and $|x|>2R/T$,
		\[
		\left|x+\frac{u-w}{T}\right|\asymp |x|,
		\]
		so that \eqref{eq:qT-bound} bounds the absolute value of the $x$-integrand
		in \eqref{eq:vT-scaled-error} by
		\[
		C_R\left(
		|x|^{2\beta+1}\mathbf 1_{\{|x|\le1\}}
		+|x|^{2\beta}\mathbf 1_{\{|x|>1\}}
		\right),
		\]
		which is integrable because $2\beta+1>-1$ and $2\beta<-1$. On $0\le u,w\le R$ and $|x|\le2R/T$, writing $d=u-w$ gives
		\[
		\int_0^R\int_0^R
		\left|x+\frac{u-w}{T}\right|^\beta\,\mathrm du\,\mathrm dw
		=\int_{-R}^R(R-|d|)\left|x+\frac dT\right|^\beta\,\mathrm dd
		\le C_R T^{-\beta},
		\]
		and hence, again by \eqref{eq:qT-bound}, the contribution of this region
		is bounded by $ C_{R}T^{-2\beta-2}$, where $2\beta+2=4H-2>0$. Finally, on $\{u+w>R\}$,
		\[
		e^{-\lambda(u+w)}\le e^{-\lambda R/2}e^{-\lambda(u+w)/2}.
		\]
		The same estimate used above, with $\lambda/2$ in place of $\lambda$, therefore shows that the contribution of this region is bounded by $Ce^{-\lambda R/2}$, uniformly in $T$. Letting $R\to\infty$ proves the required tail control. Thus the family of integrands in
		\eqref{eq:vT-scaled-error} is uniformly integrable. Combining this with
		\eqref{eq:qT-limit} yields
		\begin{align*}
			\lim_{T\to\infty}
			T^{3-4H}(v_\infty-v_T)
			&=
			\alpha_H^2
			\left(
			\int_0^\infty e^{-\overline\gamma u}\,\mathrm du
			\right)
			\left(
			\int_0^\infty e^{-\gamma w}\,\mathrm dw
			\right)
			\int_{\R}
			|x|^{2\beta}
			\bigl(1-(1-|x|)_+\bigr)\,\mathrm dx\\
			&=
			\frac{2\alpha_H^2}{|\gamma|^2}
			\left(
			\int_0^1 x^{4H-3}\,\mathrm dx
			+\int_1^\infty x^{4H-4}\,\mathrm dx
			\right)\\
			&=
			\frac{2\alpha_H^2}{|\gamma|^2}
			\left(
			\frac1{4H-2}+\frac1{3-4H}
			\right) >0.
		\end{align*}
		In particular,
		\begin{equation}\label{eq:vT-rate}
			|v_T-v_\infty|\asymp T^{4H-3}.
		\end{equation}

		Using the reversed-conjugate kernel $\psi_T^\star$ in the second factor of \eqref{eq:vT-fourfold} and putting $u=t_1-s_1$, $ w=s_2-t_2$ and $r=s_1-t_2$, one obtains
		\begin{equation}\label{eq:etaT-overlap}
			\eta_T
			=
			\alpha_H^2
			\int_{\R_+^2}e^{-\overline\gamma(u+w)}
			\int_{\R}|r+u|^\beta|r-w|^\beta q_T(r;u,w)
			\,\mathrm dr\,\mathrm du\,\mathrm dw.
		\end{equation}
		Since
		\[
		\int_{\R}|r+u|^\beta|r-w|^\beta\,\mathrm dr
		=C_\beta(u+w)^{2\beta+1},
		\]
		the limiting integral is again absolutely convergent. As $0\le q_T\le1$, the same integrable majorant applies, and dominated convergence gives
		\[
		\eta_\infty
		=\alpha_H^2\int_{\R_+^2}e^{-\overline\gamma(u+w)}
		\int_{\R}|r+u|^\beta|r-w|^\beta
		\,\mathrm dr\,\mathrm du\,\mathrm dw.
		\]
		Subtracting \eqref{eq:etaT-overlap} and then setting $r=Tx$ gives
		\begin{align*}
			T^{3-4H}(\eta_\infty-\eta_T)=
			\alpha_H^2\int_{\R_+^2}e^{-\overline\gamma(u+w)}
			\int_{\R}
			\left|x+\frac uT\right|^\beta
			\left|x-\frac wT\right|^\beta
			\bigl(1-q_T(Tx;u,w)\bigr)
			\,\mathrm dx\,\mathrm du\,\mathrm dw.
		\end{align*}
		Using the same argument as above, one shows
		\begin{align}\label{eq:etaT-upper-rate}
			\lim_{T\to\infty}
			T^{3-4H}(\eta_\infty-\eta_T)=
			\frac{2\alpha_H^2}{{\overline \gamma}^2}
			\left(
			\frac1{4H-2}+\frac1{3-4H}
			\right) \neq 0.
		\end{align}
		Finally, from \eqref{eq:CT},
		\[
		\|\Sigma_T-\Sigma_\infty\|_{\HS}^2
		=
		\frac12\left(
		|v_T-v_\infty|^2+|\eta_T-\eta_\infty|^2
		\right).
		\]
		Together with \eqref{eq:vT-rate} and \eqref{eq:etaT-upper-rate}, this proves
		both bounds in \eqref{eq:cfou-cov-rate}.
	\end{proof}
	
	The preceding lemma controls the covariance mismatch and the fourth-moment defect. It remains to identify the leading intrinsic contribution, which comes from the two third-order complex moments.
	
	\begin{lemma}\label{lem:cfou-third-moments}
		There exist constants $0<c_{\gamma,H}\le C_{\gamma,H}<\infty$ and
		$T_0\ge2$ such that, for every $T\ge T_0$,
		\begin{equation}\label{eq:cfou-third-two-sided}
			c_{\gamma,H}\tau_H(T)
			\le
			|\E[F_T^3]|+|\E[F_T^2\overline{F_T}]|
			\le
			C_{\gamma,H}\tau_H(T).
		\end{equation}
	\end{lemma}
	
	\begin{proof}
		For $x,y\in[0,T]$, set
		\begin{align*}
			K_T(x,y)
			:=
			\alpha_H\int_0^x e^{-\overline\gamma(x-u)}|u-y|^\beta\,\mathrm du,\qquad 
			L_T(x,y)
			:=
			\alpha_H\int_x^T e^{-\gamma(u-x)}|u-y|^\beta\,\mathrm du,
		\end{align*}
		and introduce their stationary counterparts
		\begin{align*}
			K(x)
			:=
			\alpha_H\int_0^\infty e^{-\overline\gamma u}|x-u|^\beta\,\mathrm du,\qquad
			L(x)
			:=
			\alpha_H\int_0^\infty e^{-\gamma u}|x+u|^\beta\,\mathrm du
			=\overline{K(-x)}.
		\end{align*}
		By the product formula \eqref{eq:complex-product-formula}, the isometry
		\eqref{eq:complex-isometry}, and \eqref{eq:fBm-inner},
		\begin{align}
			\E[F_T^3]
			&=
			\frac{2}{T^{3/2}}
			\int_{[0,T]^3}
			K_T(t_1,t_2)K_T(t_2,t_3)K_T(t_3,t_1)
			\,\mathrm dt_1 \mathrm dt_2 \mathrm dt_3,
			\label{eq:third-trace-1}\\
			\E[F_T^2\overline{F_T}]
			&=
			\frac{2}{T^{3/2}}
			\int_{[0,T]^3}
			K_T(t_1,t_2)K_T(t_2,t_3)L_T(t_3,t_1)
			\,\mathrm dt_1 \mathrm dt_2 \mathrm dt_3.
			\notag
		\end{align}
		For later use, we introduce the integral
		\[
		\mathcal G(z)
		:=
		\alpha_H\int_0^\infty e^{-\lambda u}
		\bigl(|z-u|^\beta+|z+u|^\beta\bigr)\,\mathrm du,
		\qquad z\in\R.
		\]
		For $|z|\le2$ the integral is uniformly bounded. For $|z|>2$, using the
		evenness of $\mathcal G$ and splitting the term containing $||z|-u|^\beta$
		over $[0,|z|/2]$, $[|z|/2,3|z|/2]$, and $[3|z|/2,\infty)$ gives
		\begin{equation}\label{eq:third-envelope-bound}
			\mathcal G(z)\le C_{\gamma,H}(1+|z|)^\beta,
			\qquad z\in\R.
		\end{equation}
		Here the possible singularities are locally integrable because $\beta>-1$,
		while the exponential factor controls the tails. After the changes
		of variables $u\mapsto x-u$ in $K_T$ and $u\mapsto u-x$ in $L_T$, the
		definitions immediately give
		\begin{equation}\label{eq:third-kernel-bound}
			|K_T(x,y)|+|L_T(x,y)|+|K(x-y)|+|L(x-y)|\le \mathcal G(x-y)
			\le C_{\gamma,H}(1+|x-y|)^\beta.
		\end{equation}
		Consequently,
		\begin{align*}
			|\E[F_T^3]|+|\E[F_T^2\overline{F_T}]|
			&\le
			\frac{C_{\gamma,H}}{T^{3/2}}
			\int_{[0,T]^3}
			\prod_{i=1}^3(1+|t_i-t_{i+1}|)^\beta
			\,\mathrm dt_1 \mathrm dt_2 \mathrm dt_3\\
			&\le
			\frac{C_{\gamma,H}}{T^{1/2}}
			\int_{[-T,T]^2}
			(1+|x|)^\beta(1+|y|)^\beta(1+|x-y|)^\beta
			\,\mathrm dx\,\mathrm dy,
		\end{align*}
		where $t_4=t_1$. A dyadic decomposition gives
		\begin{equation}\label{eq:triangle-integral-rate}
			\int_{[-T,T]^2}
			(1+|x|)^\beta(1+|y|)^\beta(1+|x-y|)^\beta
			\,\mathrm dx\,\mathrm dy
			\le C_H
			\begin{cases}
				1,&\beta<-2/3,\\
				\log T,&\beta=-2/3,\\
				T^{3\beta+2},&\beta>-2/3,
			\end{cases}=: \mu_{H}(T).
		\end{equation}
		This proves the upper bound in \eqref{eq:cfou-third-two-sided}.
		
		We next show that these orders are sharp. It is enough to consider
		$|\E[F_T^3]|$. We first remove the finite-boundary truncations. For $0\le x,y\le T$, writing
		$u=x+s$ and using \eqref{eq:third-envelope-bound} give
		\begin{align*}
			|K(x-y)-K_T(x,y)|
			&\le
			\alpha_H\int_x^\infty e^{-\lambda u}|x-y-u|^\beta\,\mathrm du\le e^{-\lambda x}\mathcal G(y)
			\le C_{\gamma,H}e^{-\lambda x}(1+y)^\beta.
		\end{align*}
		Hence a
		telescoping expansion,
		\eqref{eq:third-kernel-bound}, and
		\[
		\int_0^T e^{-\lambda t_1}(1+|t_3-t_1|)^\beta\,\mathrm dt_1
		\le \int_0^\infty e^{-\lambda t_1}|t_3-t_1|^\beta\,\mathrm dt_1\le C \mathcal G(t_3)  \le C(1+|t_3|)^\beta, 
		\]
		where $t_3\in \R$ and we use \eqref{eq:third-envelope-bound}, show that
		\begin{equation}\label{eq:third-stationary-replacement}
			\begin{aligned}
				&\int_{[0,T]^3}
				K_T(t_1,t_2)K_T(t_2,t_3)K_T(t_3,t_1)
				\,\mathrm dt_1\mathrm dt_2\mathrm dt_3 \\
				=&\,
				\int_{[0,T]^3}
				K(t_1-t_2)K(t_2-t_3)K(t_3-t_1)
				\,\mathrm dt_1\mathrm dt_2\mathrm dt_3
				+O(\mu_{H}(T)).
			\end{aligned}
		\end{equation}
		For example, the error coming from the first factor is bounded by
		\begin{align*}
			& \int_{[0,T]^3}
			|K_T(t_1,t_2)-K(t_1-t_2)| |K_T(t_2,t_3)| |K_T(t_3,t_1)|
			\,\mathrm dt_1\mathrm dt_2\mathrm dt_3\\
			\le&\,C\int_{[0,T]^3}e^{-\lambda t_1}(1+t_2)^\beta
			(1+|t_2-t_3|)^\beta(1+|t_3-t_1|)^\beta\,\mathrm dt_1\mathrm dt_2\mathrm dt_3\\
			\le&\,
			C\int_{[0,T]^2}(1+t_2)^\beta(1+t_3)^\beta
			(1+|t_2-t_3|)^\beta\,\mathrm dt_2\mathrm dt_3,
		\end{align*}
		and the last integral has exactly the three orders in
		\eqref{eq:triangle-integral-rate}; the other two errors are identical after
		a cyclic relabeling.
		
		It remains to determine the leading order of the stationary cube integral. Put $U(x,y):=\max\{|x|,|y|,|x+y|\}$.
		The change of variables $x=t_1-t_2$, $y=t_2-t_3$ gives
		\begin{equation}\label{eq:third-master}
			\begin{aligned}
				&\frac1T\int_{[0,T]^3}
				K(t_1-t_2)K(t_2-t_3)K(t_3-t_1)
				\,\mathrm dt_1\mathrm dt_2\mathrm dt_3\\
				=&\,
				\int_{\{U(x,y)\le T\}}
				K(x)K(y)K(-x-y)
				\left(1-\frac{U(x,y)}T\right)\,\mathrm dx\mathrm dy.
			\end{aligned}
		\end{equation}
		Indeed, after $x$ and $y$ are fixed, the remaining translation variable has
		admissible length $[T-U(x,y)]_+$.
		
		Suppose first that $H\in(1/2,2/3)$, equivalently $\beta<-2/3$.
		By \eqref{eq:third-kernel-bound} and \eqref{eq:triangle-integral-rate},
		$K(x)K(y)K(-x-y)$ is integrable on $\R^2$. Thus
		\eqref{eq:third-stationary-replacement}, \eqref{eq:third-master}, and dominated
		convergence yield
		\[
		\lim_{T\to\infty}\sqrt T\,\E[F_T^3]
		=
		2\int_{\R^2}K(x)K(y)K(-x-y)\,\mathrm dx\mathrm dy.
		\]
		The limiting constant is nonzero. Indeed, with the convention
		$\widehat f(\xi)=\int_{\R}e^{-\ii\xi x}f(x)\,\mathrm dx$,
		\[
		\widehat K(\xi)
		=c_H\frac{|\xi|^{1-2H}}{\overline\gamma+\ii\xi}.
		\]
		Fourier inversion identifies the last double integral, up to a positive normalization constant, with $\int_{\R}\widehat K(\xi)^3\,\mathrm d\xi$. Using the Laplace representation of $(\overline\gamma+\ii\xi)^{-3}$ and the
		standard cosine integral,
		\[
		\int_{\R}\widehat K(\xi)^3\,\mathrm d\xi
		=
		c_H^3\Gamma(4-6H)\Gamma(6H-1)
		\cos\bigl(\pi(2-3H)\bigr)\,\overline\gamma^{\,1-6H}\ne0.
		\]
		Hence $|\E[F_T^3]|\ge c_{\gamma,H}T^{-1/2}$ for all sufficiently large
		$T$.
		
		Next let $H\in(2/3,3/4)$, so $\beta>-2/3$. As $|x|\to\infty$, splitting the defining integral at $u=|x|/2$ and using
		the mean value theorem on the near part gives
		\begin{equation}\label{eq:K-tail}
			K(x)=\frac{\alpha_H}{\overline\gamma}|x|^\beta
			+O(|x|^{\beta-1}).
		\end{equation} 
		Scale
		$x=T\xi$, $y=T\eta$ in \eqref{eq:third-master}. By \eqref{eq:K-tail},
		\[
		T^{-\beta}K(T\xi)\rightarrow
		\frac{\alpha_H}{\overline\gamma}|\xi|^\beta
		\]
		for almost every $\xi$. Moreover, \eqref{eq:third-kernel-bound} gives
		$T^{-\beta}|K(T\xi)|\le C|\xi|^\beta$ for $\xi\ne0$, and by \eqref{eq:triangle-integral-rate}
		$|\xi|^\beta|\eta|^\beta|\xi+\eta|^\beta$ is integrable on
		$\{U(\xi,\eta)\le1\}$ precisely because $\beta>-2/3$. Therefore dominated convergence, together with \eqref{eq:third-stationary-replacement}, gives
		\begin{align*}
			T^{9/2-6H}\E[F_T^3]
			\rightarrow
			2\left(\frac{\alpha_H}{\overline\gamma}\right)^3
			\int_{\{U(\xi,\eta)\le1\}}
			|\xi|^\beta|\eta|^\beta|\xi+\eta|^\beta
			(1-U(\xi,\eta))\,\mathrm d\xi\mathrm d\eta.
		\end{align*}
		The integral on the right is finite and strictly positive, so the limiting
		constant is nonzero. This proves the lower bound in this range.

		Finally, let $H=2/3$, so that $\beta=-2/3$. Put $R(x,y):=\min\{|x|,|y|,|x+y|\}$. We start from \eqref{eq:third-master} and split the domain
		$\{U(x,y)\le T\}$ according to whether all three quantities
		$|x|$, $|y|$, and $|x+y|$ are in the tail region:
		\begin{align*}
			\{U(x,y)\le T\}=
			\{U(x,y)\le T,\ R(x,y)<1\}\cup
			\{1\le U(x,y)\le T,\ R(x,y)\ge1\}.
		\end{align*}
		We first consider the region where $R(x,y)<1$. By
		\eqref{eq:third-kernel-bound},
		\[
		|K(x)K(y)K(-x-y)|
		\le
		C(1+|x|)^{-2/3}(1+|y|)^{-2/3}
		(1+|x+y|)^{-2/3}.
		\]
		For instance, on $\{|x|<1\}$ one has
		$1+|x+y|\asymp1+|y|$, uniformly in $x$, and hence
		\[
		\int_{\{|x|<1\}}
		|K(x)K(y)K(-x-y)|\,\mathrm dx\,\mathrm dy
		\le
		C\int_{\R}(1+|y|)^{-4/3}\,\mathrm dy
		<\infty.
		\]
		The regions $\{|y|<1\}$ and $\{|x+y|<1\}$ are treated in the same
		way. Since $0\le1-U(x,y)/T\le1$, the contribution of
		$\{U(x,y)\le T,\ R(x,y)<1\}$ to \eqref{eq:third-master} is therefore
		$O(1)$, uniformly in $T$. We now turn to the region $\{1\le U(x,y)\le T,\ R(x,y)\ge1\}$. Here the tail expansion \eqref{eq:K-tail} can be applied simultaneously
		to all three factors, giving
		\begin{align*}
			&K(x)K(y)K(-x-y)
			=
			\left(\frac{\alpha_H}{\overline\gamma}\right)^3
			|x|^{-2/3}|y|^{-2/3}|x+y|^{-2/3}\\
			+&\,
			O\Big(
			|x|^{-5/3}|y|^{-2/3}|x+y|^{-2/3}+
			|x|^{-2/3}|y|^{-5/3}|x+y|^{-2/3}+
			|x|^{-2/3}|y|^{-2/3}|x+y|^{-5/3}
			\Big).
		\end{align*}
		The remainder is integrable on this region. Indeed, for the first term,
		$$
		\int_{\{|y|\ge1,\ |x+y|\ge1\}}
		|y|^{-2/3}|x+y|^{-2/3}\,\mathrm dy
		\le
		C(1+|x|)^{-1/3},
		$$
		and therefore
		$$
		\int_{\{R(x,y)\ge1\}}
		|x|^{-5/3}|y|^{-2/3}|x+y|^{-2/3}
		\,\mathrm dx\mathrm dy
		<\infty.
		$$
		The other two terms are treated in the same way. Since
		$|x|^{-2/3}|y|^{-2/3}|x+y|^{-2/3}$ is homogeneous of degree $-2$
		and $U$ is homogeneous of degree one, the integral over each dyadic annulus $\{2^j\le U(x,y)\le2^{j+1}\}$, $j\in \N$, has the same strictly positive value. Summing the annuli up to scale $T$ gives the logarithmic term, while the contribution of the factor $U(x,y)/T$ is $O(1)$. Hence there exists $c_0>0$ such that
		\[
		\int_{\{1\le U(x,y)\le T\}}
		|x|^{-2/3}|y|^{-2/3}|x+y|^{-2/3}
		\left(1-\frac{U(x,y)}T\right)\,\mathrm dx\mathrm dy
		=c_0\log T+O(1).
		\]
		Moreover,
		\[
		\int_{\{1\le U(x,y)\le T,\ R(x,y)<1\}}
		|x|^{-2/3}|y|^{-2/3}|x+y|^{-2/3}\,\mathrm dx\,\mathrm dy
		=O(1),
		\]
		uniformly in $T$. Indeed, by symmetry it suffices to consider the region
		$\{|x|<1,\ 1\le U(x,y)\le T\}$. If $|y|>2$, then
		$|x+y|\asymp |y|$ uniformly in $|x|<1$, and hence
		\[
		|x|^{-2/3}|y|^{-2/3}|x+y|^{-2/3}
		\le
		C|x|^{-2/3}|y|^{-4/3},
		\]
		which is integrable on $\{|x|<1,\ |y|>2\}$. On the remaining bounded
		region $\{|x|<1,\ |y|\le2,\ U(x,y)\ge1\}$, the singularities along
		$x=0$, $y=0$, and $x+y=0$ are locally integrable since $-2/3>-1$.
		Thus removing the part where $R(x,y)<1$ from $\{1\le U(x,y)\le T\} $ changes the dyadic sum by only $O(1)$. 
		Consequently \eqref{eq:third-master} yields
		\[
		\frac1T\int_{[0,T]^3}
		K(t_1-t_2)K(t_2-t_3)K(t_3-t_1)\,\mathrm dt_1\mathrm dt_2\mathrm dt_3
		=
		\left(\frac{\alpha_H}{\overline\gamma}\right)^3c_0\log T+O(1).
		\]
		Combining this with \eqref{eq:third-trace-1} and
		\eqref{eq:third-stationary-replacement},
		\[
		\E[F_T^3]
		=
		2\left(\frac{\alpha_H}{\overline\gamma}\right)^3
		c_0T^{-1/2}\log T+O(T^{-1/2}),
		\]
		and the leading coefficient is nonzero. This proves the lower bound
		and completes the proof.
	\end{proof}

	With these estimates in hand, we can now prove the stated approximation rates.
	
	\begin{proof}[Proof of Theorem~\ref{thm:cfou-density}]
		By \eqref{eq:cfou-CLT} and \eqref{eq:cfou-cov-rate},
		\[
		\Sigma_T\rightarrow \Sigma_\infty,
		\qquad
		\Sigma_\infty\text{ is positive definite},
		\qquad
		\widetilde F_T\xrightarrow{\mathrm{law}}N_2(0,\Sigma_\infty).
		\]
		Applying Theorem~\ref{thm:general-optimal} along an arbitrary sequence
		$T_n\to\infty$ and using a standard sequential contradiction argument,
		we obtain, for all sufficiently large $T$,
		\[
		\|p_T-\phi_T\|_{W^{m,r}(\R^2)}
		\asymp
		d_\ast\bigl(\widetilde F_T,N_2(0,\Sigma_T)\bigr)
		\asymp
		M_{\C}(F_T).
		\]
		Combining Lemma~\ref{lem:cfou-third-moments} with
		\eqref{eq:cfou-defect-rate} and \eqref{eq:rho-tau-comparison}, we obtain,
		for all sufficiently large $T$,
		\begin{equation}\label{eq:MC-rate-cfou}
			M_{\C}(F_T)\asymp\tau_H(T).
		\end{equation}
		This proves \eqref{eq:cfou-optimal-CT}. Applying Theorem~\ref{thm:intro-main} in the same way gives
		\[
		\|p_T-\phi_\infty\|_{W^{m,r}(\R^2)}
		\asymp
		d_\ast\bigl(\widetilde F_T,N_2(0,\Sigma_\infty)\bigr)
		\asymp
		M_{\C}(F_T)+\|\Sigma_T-\Sigma_\infty\|_{\HS}.
		\]
		If $H\in(1/2,5/8]$, then \eqref{eq:MC-rate-cfou} gives
		$M_{\C}(F_T)\asymp T^{-1/2}=\rho_H(T)$, while
		\eqref{eq:cfou-cov-rate} gives
		$\|\Sigma_T-\Sigma_\infty\|_{\HS}\le C T^{4H-3}\le C T^{-1/2}$.
		If $H\in(5/8,3/4)$, then \eqref{eq:cfou-cov-rate} gives $\|\Sigma_T-\Sigma_\infty\|_{\HS}\asymp T^{4H-3}=\rho_H(T)$,
		and \eqref{eq:rho-tau-comparison} gives
		$M_{\C}(F_T)\asymp\tau_H(T)\le C\rho_H(T)$.
		Thus in both cases
		\[
		M_{\C}(F_T)+\|\Sigma_T-\Sigma_\infty\|_{\HS}\asymp\rho_H(T),
		\]
		which proves \eqref{eq:cfou-explicit-Cinfty}.
	\end{proof}

	\begingroup
	\setlength{\emergencystretch}{1.5em}
	\hbadness=2000
	\printbibliography
	\endgroup

\end{document}